\documentclass[leqno,12pt]{article}

\usepackage[all]{xy}
\usepackage{amsthm}
\usepackage{amssymb} 
\usepackage{amsmath,amscd} 
\usepackage{mathrsfs}
\usepackage{color}
\usepackage{enumerate}

\def\asid{\textsf{ASID}}

\def\GrMod{\operatorname{\mathsf{GrMod}}}

\def\turn!{\textup{!`}}

\def\op{\textup{op}}

\def\pd{\mathop{\mathrm{pd}}\limits}
\def\grpd{\mathop{\mathrm{gr.pd}}\limits}
\def\injdim{\mathop{\mathrm{id}}\limits}
\def\grinjdim{\mathop{\mathrm{gr.id}}\limits}

\def\tot{\operatorname{\mathsf{tot}}}

\def\grmod{\operatorname{mod}^{\Bbb{Z}}}

\def\mrb{\mathrm{b}}

\def\grSing{\operatorname{Sing}^{\ZZ}}

\def\stabgrCM{\operatorname{\underline{\mathsf{CM}}}^{\Bbb{Z}}}

\def\Tor{\operatorname{Tor}}

\def\kk{{\mathbf k}}

\def\NN{{\Bbb N}}

\def\ZZ{{\Bbb Z}}

\def\cA{{\cal A}}

\def\cX{{\cal X}}
\def\cY{{\cal Y}}

\def\sfC{{\mathsf{C}}}
\def\sfD{{\mathsf{D}}}

\def\sfK{{\mathsf{K}}}

\def\sfU{{\mathsf{U}}}

\def\sfj{{\mathsf{j}}}
\def\sfq{{\mathsf{q}}}

\def\tuD{{\textup{D}}}

\def\tuH{{\textup{H}}}

\def\frki{{\mathfrak{i}}}

\def\frkj{{\mathfrak{j}}}

\def\frkp{{\mathfrak{p}}}
\def\frkq{{\mathfrak{q}}}

\def\frks{{\mathfrak{s}}}

\def\frkt{{\mathfrak{t}}}

\def\id{\operatorname{id}}
\def\op{\operatorname{op}}

\def\mod{\operatorname{mod}}
\def\Mod{\operatorname{Mod}}
\def\GrMod{\operatorname{Mod}^{\mathbb{Z}}}

\def\Ker{\mathop{\mathrm{Ker}}\nolimits}

\def\proj{\operatorname{proj}}

\def\grproj{\operatorname{proj}^{\mathbb{Z}}} 
 
\def\GrProj{\operatorname{Proj}^{\mathbb{Z}}} 
\def\GrInj{\operatorname{Inj}^{\mathbb{Z}}}

\def\Proj{\operatorname{Proj}} 
\def\Inj{\operatorname{Inj}}

\def\Cok{\operatorname{Cok}}
\def\Coker{\operatorname{Cok}}

\def\Hom{\operatorname{Hom}}
\def\grHom{\operatorname{HOM}}

\def\Ext{\operatorname{Ext}}
\def\grExt{\operatorname{EXT}}

\def\gldim{\operatorname{gldim}}

\def\grgldim{\operatorname{grgldim}}

\newcommand{\RHom}{\operatorname{\Bbb{R}Hom}}
\def\grRHom{\operatorname{\Bbb{R}HOM}}
\newcommand{\lotimes}{\otimes^{\Bbb{L}}}

\newcommand{\cone}{\operatorname{\mathsf{cn}}}

\def\RHom{\operatorname{\mathbb{R}Hom}}
\def\grRHom{\operatorname{\mathbb{R}HOM}}

\newtheorem{lemma}{Lemma}[section]
\newtheorem{proposition}[lemma]{Proposition}
\newtheorem{theorem}[lemma]{Theorem}
\newtheorem{corollary}[lemma]{Corollary}

\newtheorem{remark}[lemma]{Remark}

\theoremstyle{definition}

\newtheorem{definition}[lemma]{Definition}

\theoremstyle{remark}

\title{
Homological dimension formulas for trivial extension algebras
\\
{\small
Dedicated to I. Reiten's 75 birthday
}
}

\author{Hiroyuki Minamoto and Kota Yamaura}

\begin{document}

\maketitle

\begin{abstract}
Let $A= \Lambda \oplus C$ be a trivial extension algebra. 
The aim of this paper is to establish formulas for the projective dimension and the injective dimension for a certain class of $A$-modules 
which is expressed by using the derived functors $- \lotimes_{\Lambda}C$ and $\RHom_{\Lambda}(C, -)$. 
Consequently, 
we obtain  formulas for the global dimension of $A$, 
which gives a modern expression of  the classical formula for the global dimension  by Palmer-Roos and L\"ofwall 
that is written in complicated classical derived functors.

The main application of the formulas is to give 
a necessary and sufficient condition for $A$ to be  an  Iwanaga-Gorenstein algebra. 

We also give a description of the kernel $\Ker \varpi$ 
of the canonical functor $\varpi: \sfD^{\mrb}(\mod \Lambda) \to \grSing A$  
 in the  case $\pd C < \infty$. 
\end{abstract}

\tableofcontents

\section{Introduction}\label{Introduction}

Throughout the paper $\kk$ denotes a commutative ring. 
An algebra $\Lambda$ is always $\kk$-algebra 
and a $\Lambda$-$\Lambda$-bimodule $C$ is always assumed to be $\kk$-central. 
Recall that the \textit{trivial extension algebra} $A = \Lambda \oplus C$ is a direct sum $\Lambda \oplus C$ 
equipped with the multiplication 
\[
(r, c) (s,d) := (rs , rd+ cs) \ \ \ \ ( r,s \in \Lambda, \ \ c, d \in C).
\] 
Since  a trivial extension algebra is   one of fundamental construction, 
it has been extensively studied from every aspect 
and homological dimension is no exception.  

For instance,  in \cite{Chase}, 
Chase raised a problem of determining the global dimension 
of an upper triangular matrix algebras  $A  =\begin{pmatrix} \Lambda_{0} & C \\ 0 & \Lambda_{1} \end{pmatrix}$ 
which is an example of a trivial extension algebra,  in terms of $\Lambda_{0}, \Lambda_{1}$ and $C$.

The global dimension of general trivial extension algebra $A = \Lambda \oplus C$ 
had been studied by Fossum-Griffith-Reiten \cite{FGR}, Reiten \cite{Reiten:Thesis}, Palmer-Roos \cite{Palmer-Roos}. 
Finally, L\"ofwall \cite{Lofwall} gave a general formula for the global dimension of $A$ in terms of $\Lambda$ and $C$ 
by using ``multiple $\Tor$" introduced by  Palmer-Roos \cite{Palmer-Roos}. 
Thus, in particular, Chase's problem was solved. 
(For the historical background we refer the readers to  \cite[Section 4]{FGR}, \cite[Introduction]{Palmer-Roos}.)

However, 
the methods ``multiple $\Tor$" for the Palmer-Roos-L\"ofwall formula
was so complicated that 
the formula has never got attention which it ought to deserve.  
For instance, in the studies of homological dimensions of upper triangular algebras 
(e.g.,\cite{Asadollahi-Salarian,Chen scs,ECIT,Sakano}) 
there have been no attempt to generalize the formula 
in such a way as  to be applicable for each problem.

In this paper we establish  formulas for homological dimensions of a class of $A$-modules  
by using homological dimension of objects of the derived category $\sfD(\Mod \Lambda)$ introduced by Avramov-Foxby 
\cite{Avramov-Foxby}. 
As a corollary, 
we obtain  formulas for the global dimension of $A$, 
which gives a modern expression of  the Palmer-Roos-L\"ofwall formula.

We note that the formulas involve the  iterated derived  tensor product $C^{a}$ of $C$, 
where  for  $a \in \NN$, we set\footnote{
However, we remark that there is a subtlety about derived tensor product of bimodules. 
For this see Remark \ref{bimodule remark projective} and Remark \ref{bimodule remark injective}.}
\[
C^{a} := 
\begin{cases} 
C\lotimes_{\Lambda} C \lotimes_{\Lambda}  \cdots \lotimes_{\Lambda} C  \ \ (a\textup{-factors}) & a >0,  \\
\Lambda & a = 0.
\end{cases}
\]

In the rest of Introduction, we explain the results of this paper by only focusing on injective dimensions. 
A key technique   
is the use of the grading with which a trivial extension algebra $A = \Lambda \oplus C$ is canonically equipped. 
Namely, $\deg\Lambda = 0, \ \deg C= 1$.  
Let $M$ be a graded $A$-module concentrated in degree $i= 0, 1$, i.e., $M_{i} = 0 $ for $i \neq 0,1$. 
Then, we will observe in Proposition \ref{finitely graded injective dimension lemma} that 
its (ungraded) injective dimension and graded injective dimension coincide.
\[
\injdim _{A} M = \grinjdim_{A} M.  
\]
By this fact, we can pass to the study  of the graded injective dimension of $M$. 
It is an analysis of  a graded injective resolution of $M$ that naturally leads to the iterated derived tensor product $C^{a}$. 
To state our injective dimension formula, we  use  the derived coaction morphism $\Theta_{M}^{0}$ of $M$, 
that is, the  morphism induced from the graded $A$-module structure on $M=M_{0} \oplus M_{1}$.  
\[
\Theta_{M}^{a}: M_{0} \to \RHom_{\Lambda}(C, M_{1}). 
\]
For simplicity we set   $\Theta_{M}^{a}:= \RHom_{\Lambda}(C^{a}, \Theta_{M}^{0})$. 
\[
\Theta_{M}^{a} : \RHom_{\Lambda}(C^{a}, M_{0}) \to \RHom_{\Lambda}(C^{a+1}, M_{1}).
\]
We can now formulate our injective dimension formula. 

\begin{theorem}[{Theorem \ref{injective dimension formula}}]\label{introduction theorem 1}
Let $M$ be a graded $A$-module such that $M_{i} = 0$ for $i \neq 0,1$. 
Then, 
\[
\injdim_{A} M = \grinjdim_{A} M = \sup\{\injdim_{\Lambda}  M_{1}, \ \injdim_{\Lambda} ( \cone \Theta_{M}^{a}) + a+ 1 \mid a \geq 0\}
\]
where $\cone \Theta_{M}^{a}$ denote the cone of the morphism $\Theta_{M}^{a}$ in the derived category $\sfD(\Mod \Lambda)$. 
\end{theorem}
We note that this formula and  a projective version given in Theorem \ref{projective dimension formula} 
is established for not only a graded $A$-module  but also an object of the derived category $\sfD(\GrMod A)$ 
satisfying the same condition. 

In the case where $M$ is concentrated in degree $0$, the formula is simplified as in Corollary \ref{injective dimension formula corollary}.  
We remark that such a graded $A$-module is nothing but a $\Lambda$-module 
regarded as an $A$-module via the augmentation map $\mathsf{aug}: A \to \Lambda, \mathsf{aug}(r,c) := r$.  
The global dimension of $A$ can be  measured by such modules. 

\begin{corollary}[{Corollary \ref{injective dimension formula corollary 2}}]\label{Introduction theorem 2}
\[
\gldim A = \sup \{ \injdim_{\Lambda} \RHom_{\Lambda}(C^{a}, M)  + a \mid M \in \Mod \Lambda, \ a \geq 0 \}.
\]
\end{corollary}
We remark that the projective dimension version of the above formula which is   given in Corollary \ref{projective dimension formula corollary 2} 
is essentially the same with the result of  Palmer-Roos  and L\"ofwall given in the aforementioned papers. 
We also remark that it seems that the above global dimension formula can be proved by their methods.

However,   our main  application of Theorem \ref{introduction theorem 1}, which is a criterion of finiteness of self-injective dimension of $A$, 
seems  to be hard to be  achieved by their methods. 
The graded component of the graded $A$-module $A_{A}$ are $A_{0} = \Lambda, \ A_{1} = C$ and $ A_{i} = 0$ for $i \neq 0,1$ 
and the derived coaction morphism $\Theta_{A}^{0}$  is 
the morphism 
\[
\lambda_{r}: \Lambda \to \RHom_{\Lambda}(C,C)
\]
induced from the left 
multiplication map $\tilde{\lambda}_{r}(r): C \to C, (\tilde{\lambda}_{r}(r))(c) := sc \ (r \in \Lambda)$ 
(where the suffix $r$ of $\lambda_{r}$ indicate that this relates to the right self-injective dimension). 
As a consequence of Theorem \ref{introduction theorem 1} we obtain the following criterion.

\begin{theorem}[{Theorem \ref{right asid theorem}}]\label{Introduction theorem 3}
The following conditions are equivalent: 
\begin{enumerate}[(1)]
\item $\injdim A_{A} < \infty$.  

\item the following conditions are satisfied: 

\begin{enumerate}[ {Right} $\asid$ 1.]
\item $\injdim_{\Lambda} C < \infty$. 

\item $\injdim_{\Lambda} \cone (\RHom_{\Lambda}(C^{a}, \lambda_{r} )) < \infty $ 
for $a \geq 0$. 

\item  The morphism $\RHom_{\Lambda}(C^{a}, \lambda_{r} )$ is an isomorphism for $a \gg 0$. 
\end{enumerate}
Here ``\asid" is an abbreviation of ``attaching self-injective dimension".
 
\end{enumerate}
\end{theorem}

We would like to mention that the results of this paper came out of  the study of finitely graded IG-algebras. 
Recall that a graded  algebra is called Iwanaga-Gorenstein (IG) if it is graded Noetherian on both sides and 
has finite graded self-injective dimension on both sides. 
Representation theory of (graded and  ungraded) IG-algebra was 
initiated by Auslander-Reiten \cite{AR}, Happel \cite{Happel}  and Buchweitz \cite{Buchweitz},  
has been studied by many researchers 
and is recently getting interest from other areas  \cite{ART,BIRS,GLS,Keller-Reiten,Kimura 1,Kimura 2}. 

As is explained in Section \ref{quasi-Veronese}, 
every finitely graded algebra is graded Morita equivalent to a trivial extension algebra. 
Hence, representation theoretic study of finitely graded algebras can be reduced to that of trivial extension algebras.  
By Theorem \ref{Introduction theorem 3}, 
 a trivial extension algebra $A = \Lambda \oplus C$ is IG 
if and only if 
$C$ satisfies the right \asid \ conditions and the left version of them, the left \asid \ conditions.  
As is proved in Proposition \ref{reduction of asid conditions}, 
if $\Lambda$ is IG, then the \asid \ conditions are simplified. 
In the subsequent paper \cite{higehaji},  
we prove that 
in the case where $\Lambda$ is IG, 
the right and left  \asid  \ conditions  has a categorical interpretation.   
Using this interpretation we establish a relationship between the derived category $\sfD^{\mrb}(\mod \Lambda)$ 
of $\Lambda$ and the stable category $\stabgrCM A$ of graded Cohen-Macaulay modules over $A$ 
and provide several applications. 

In \cite{anodai},   
we introduce a new class of finitely graded IG-algebra called \textit{homologically well-graded (hwg) IG-algebra}, 
and show that it posses nice characterizations from several view points. 
In particular, we  characterize the condition  that $A = \Lambda \oplus C$ is hwg  
in terms of the right and left asid numbers  $\alpha_{r}, \alpha_{\ell}$ which is introduced in Definition \ref{asid number definition}.

The paper is organized as follows. 
In Section \ref{Homological algebra of finitely graded algebras}, 
we discuss homological algebra of finitely graded algebras. 
In particular, we study relationships   
between  graded homological dimensions  and ungraded homological dimensions.

In Section \ref{Homological dimension},
we recall the projective dimension and the injective dimension for unbounded complexes introduced by  Avramov-Foxby \cite{Avramov-Foxby}.

In Section \ref{A criterion of perfectness}, we establish  the formula for the projective dimensions. 
The key tool is the decomposition of complexes of graded projective $A$-modules 
according to the degree of generators introduced by Orlov \cite{Orlov}. 
 Analyzing the decomposition, we relate a projective resolution of $M$ as graded $A$-modules 
 with that of $M$ as a $\Lambda$-module via the iterated derived tensor products $C^{a}$.

 Among other things, 
in  Corollary \ref{description of kernel},   
we give a description of the kernel $\Ker \varpi$ of the 
 canonical functor $\varpi: \sfD^{\mrb}(\mod \Lambda) \to \grSing A$ 
 where $\grSing A$ is the graded singular derived category of $A$. 
This result also plays an important role in \cite{higehaji}.

 In Section \ref{A criterion for finiteness of injective dimension}, 
 we establish the formula for the injective dimension. 
 The key tool is the decomposition of complexes of graded injective $A$-modules.  
Analyzing the decomposition,  
we  relate an injective  resolution of $M$ as graded $A$-modules 
 with that of $M$ as a $\Lambda$-module via the derived Hom functors  $\RHom_{\Lambda}(C^{a},-)$. 
We give a criterion that $A= \Lambda \oplus C$ is IG in terms of $\Lambda$ and $C$. 
We see that if $\Lambda$ is IG, then the condition is simplified.

In Section \ref{Upper triangular matrix algebras}, 
we discuss an upper triangular matrix algebra $A =\begin{pmatrix} \Lambda_{0} & C \\ 0 & \Lambda_{1} \end{pmatrix}$, 
which is an example of  a trivial extension  algebra.   
Although,  Chase's problem was already  solved as a corollary of the main result of \cite{Palmer-Roos} and  \cite{Lofwall}, 
we give our own answer.  
Since every $A$-module $M$ has a canonical grading such that $M_{i} = 0$ for $i \neq 0,1$, 
we obtain formulas of the projective dimension and the injective dimension of $M$ 
in terms of $\Lambda_{0}$, $\Lambda_{1}$ and $C$.
Moreover, as immediate corollaries,  
we deduce other known results concerning on upper triangular algebras.

\subsection{Notation and convention}

Let $\Lambda$ be an algebra. 
Unless otherwise stated, the word  ``$\Lambda$-modules" means right $\Lambda$-modules.  
We denote by $\Mod \Lambda$ the category of $\Lambda$-modules. 
We denote by $\Proj \Lambda$ (resp. $\Inj \Lambda$) 
the full subcategory of projective (resp. injective) $\Lambda$-modules. 
We denote  by $\proj  \Lambda$ the full subcategory of finitely generated projective modules.

We denote the opposite algebra  by $\Lambda^{\op}$. 
We identify left $\Lambda$-modules with (right) $\Lambda^{\op}$-modules.   
A $\Lambda$-$\Lambda$-bimodule $D$ is always assumed to be $\kk$-central, 
i.e., $ad = da $ for $d \in D, \ a \in \kk$. 
For a $\Lambda$-$\Lambda$-bimodule $D$, 
we denote by $D_{\Lambda}$ and ${}_{\Lambda} D$ 
the underlying right  and left $\Lambda$-modules respectively. 
So for example, 
$\injdim_{\Lambda^{\op}}  {}_{\Lambda} D$ denotes the injective dimension of $D$ regarded a left $\Lambda$-module.

In the paper, the degree of graded modules $M$ is usually  denoted by the characters $i,j ,\dots$. 
The degree, which is called the cohomological degree, of complexes $X$  is usually denoted by the characters $m,n, \dots$.

\newpage

\vspace{10pt}
\noindent
\textbf{Acknowledgment}


The authors thank M. Sato for pointing out the reference \cite{Sakano}. 
They also thank O. Iyama for the comments on the fist draft of the paper.  
The first author  was partially  supported by JSPS KAKENHI Grant Number 26610009.
The second author  was partially  supported by JSPS KAKENHI Grant Number 26800007.


\section{Homological algebra of finitely graded algebras}\label{Homological algebra of finitely graded algebras}

In this paper, a graded algebra is always a non-negatively graded algebra $A = \bigoplus_{i \geq 0} A_{i}$. 
A graded algebra $A=\bigoplus_{i \geq 0}A_{i}$ is called \textit{finitely graded} if $A_{i}= 0$ for $i \gg 0$. 
In this Section \ref{Homological algebra of finitely graded algebras}, 
we collect basic facts about homological algebra of  finitely graded algebras.  
Another aim is to introduce constructions $\frkp_{i} P, \  \frkt_{i} P, \ \frki_{i} I , \ \frks_{i} I $  
which play a central role in this paper.

\subsection{Notation and convention for graded algebras and graded modules}

Let $A = \bigoplus_{i \geq 0} A_{i}$ be a graded algebra. 
We denote by $\GrMod A$ the category of graded (right) $A$-modules $M= \bigoplus_{i \in \ZZ} M_{i}$ 
and graded $A$-module homomorphisms $f: M \to N$, 
which, by definition, preserves degree of $M$ and $N$, 
i.e., $f(M_{i}) \subset N_{i}$ for $i \in \ZZ$.  
We denote by $\GrProj A$ (resp. $\GrInj A$) the full subcategory of graded projective (resp. graded injective) modules. 
We denote by $\grproj  A$ the full subcategory of finitely generated graded projective modules.

For a graded $A$-module $M$ and an integer $j \in \ZZ$, 
we define the shift $M(j) \in \GrMod A$ by $(M(j))_{i} = M_{i+j}$.  
We define the truncation $M_{\geq  j}$ by 
$(M_{\geq  j})_{i} =M_{i} \ (i \geq j), \ \ (M_{\geq j})_{i} = 0  \ ( i < j)$. 
We set $M_{ < j} := M/M_{\geq j}$ so that we have an exact sequence $ 0\to M_{\geq j} \to M \to M_{< j} \to 0$.

For $M, N \in \GrMod A, \ n \in \NN$ and $i \in \ZZ$,  
we set $\grExt_{A}^{n}(M,N)_{i}:= \Ext_{\GrMod A}^{n}(M,N(i))$ 
and 
\[
\grExt_{A}^{n}(M, N) := \bigoplus_{i \in \ZZ} \grExt_{A}^{n}(M, N)_{i} = \bigoplus_{i \in \ZZ} \Ext_{\GrMod A}^{n}(M, N(i)). 
\]
We note the obvious equation $\grHom_{A}(M,N)_{0} = \Hom_{\GrMod A}(M,N)$. 
 We denote by $\Hom_{A}(M,N)$ the $\Hom$-space as ungraded $A$-modules. 
 We note that there exists  the canonical map $\grHom_{A}(M, N) \to \Hom_{A}(M,N)$ 
and it become an isomorphism if $M$ is finitely generated. 

   For further details of graded algebras and graded modules  we refer the readers to \cite{NV:Graded and Filtered}.

\subsection{Graded projective dimension and ungraded projective dimension  of graded $A$-modules}
In the rest of  this section, 
 $A = \bigoplus_{i= 0}^{\ell} A_{i}$ is a finitely graded algebra. 
 We note that $\ell$ is a natural number such that $A_{i} = 0$ for $i \geq \ell+1$ 
 and that it is not necessary to assume $A_{\ell} \neq 0$.  
For notational simplicity we set $\Lambda := A_{0}$.

The following lemma can be easily checked and  is left to the readers. 
\begin{lemma}\label{basic projective lemma}
\begin{enumerate}[(1)]
\item 
Let $P$ be a graded projective $A$-module. 
Then, for all $i \in \ZZ$, the module $(P \otimes_{A} \Lambda)_{i}$ is a projective $\Lambda$-module. 

\item 
Let $Q$ be a projective $\Lambda$-module. 
Then, the graded module $Q \otimes_{\Lambda} A$ is a graded projective $A$-module. 
\end{enumerate}
\end{lemma}

Since the grading of $A$ is finite, 
the following Nakayama type Lemma follows. 

\begin{lemma}\label{Nakayama lemma} 
Let $M$ be a graded $A$-module. 
Then $M = 0$ if and only if $M \otimes_{A} \Lambda = 0$. 
\end{lemma}

For an integer $i \in \ZZ$, 
we denote  by
$\frkp_{i}: \GrProj A \to \Proj \Lambda$ the functor $\frkp_{i} P := (P \otimes_{A} \Lambda)_{i}$. 
We define a graded $A$-module  $\frkt_{i} P$ to be $\frkt_{i}P := (\frkp_{i}P ) \otimes_{\Lambda} A (-i)$.  
We may consider $\frkp_{i} P$ as the space of generators of $P$ having degree $i$.

\begin{lemma}[\textup{cf. \cite[Proposition 2.6]{MM}}]\label{projective decomposition lemma} 
Let $P$ be an object of $\GrProj A$.   
Then,  we have an isomorphism of graded $A$-modules 
\begin{equation}
 P  \cong \bigoplus_{ i\in \ZZ} \frkt_{i} P. 
\end{equation} 
\end{lemma}  

\begin{proof}
For simplicity, we set $\tilde{P} :=\bigoplus_{ i\in \ZZ} \frkt_{i} P$. 
We remark that $\tilde{P} = (P\otimes_{A} \Lambda) \otimes_{\Lambda} A$ 
and it is a graded projective $A$-module by Lemma \ref{basic projective lemma}.  
There are canonical surjective graded $A$-module homomorphisms $p: P \to P \otimes_{A} \Lambda$ 
and $\tilde{p} :  \tilde{P} \to P \otimes_{A} \Lambda$. 
Since $P$ is a  graded projective $A$-module, there exists a graded $A$-module homomorphism 
$f: P \to \tilde{P}$ such that $f \otimes \Lambda$ is an isomorphism. 
Since $\tilde{P}$ is a flat as an ungraded $A$-module, we have $(\Cok f) \otimes_{A} \Lambda = 0, \  (\Ker f ) \otimes_{A} \Lambda = 0$. 
Therefore $\Cok f = 0, \ \Ker f = 0$ by Lemma \ref{Nakayama lemma}.  
\end{proof} 

Later we will use the following corollary.

\begin{corollary}\label{projective Hom decomposition corollary}
We have the following isomorphism for $P \in \GrProj A$ and $M \in \GrMod A$ 
\[
\Hom_{\GrMod A}(P,M) \cong \prod_{i\in \ZZ} \Hom_{\Lambda}(\frkp_{i}P, M_{i}). 
\]
\end{corollary}

\begin{proof}
We define a map 
$\Phi: \Hom_{\GrMod A}(P,M) \to \prod_{i\in \ZZ} \Hom_{\Lambda}(\frkp_{i}P, M_{i})$ in the following way. 
Let $f: P \to M$ be a graded $A$-module homomorphism. 
For $i \in \ZZ$, 
we define the $i$-th component $\Phi(f)_{i}: \frkp_{i} P \to  M_{i}$ of $\Phi(f)$ to be 
the composite map 
\[
\Phi(f)_{i} : \frkp_{i} P \xrightarrow{ \ \mathsf{in}_{i} \ } P \xrightarrow{ \ f \ } M \xrightarrow{ \ \mathsf{ pr }_{i} \ } M_{i} 
\]
where $\mathsf{in}_{i}$ is a canonical inclusion and $\mathsf{pr}_{i}$ is a canonical projection.

We define a map 
$\Psi: \prod_{i \in \ZZ} \Hom_{\Lambda}(\frkp_{i}P, M_{i}) \to \Hom_{\GrMod A}(P,M)$ 
in the following way. 
Let $g= (g_{i})_{i \in \ZZ}$ be a collection of $\Lambda$-module homomorphisms 
$g_{i}: \frkp_{i} P \to M_{i}$. 
Observe  that $g_{i}$ extends to a graded $A$-module homomorphism 
$\hat{g}_{i}: \frkt_{i} P = \frkp_{i} P \otimes_{\Lambda} A \to M$. 
We define $\Psi(g): P \to M$ to be the sum 
\[
\Psi(g) :=\sum_{ i \in \ZZ} \hat{g}_{i} : P \cong \bigoplus_{i \in \ZZ} \frkt_{i} P\xrightarrow{ \ \  \ \ }  M. 
\]
We can check that $\Phi$ and $\Psi$ are the inverse map to each other.
\end{proof}

For a graded module $M$, 
its graded projective dimension and its (ungraded) projective dimension coincide. 
We remark that this statement is true for any $\ZZ$-graded algebra $A$ (\cite[I.3.3.12]{NV:Graded and Filtered}).
For convenience of the readers, we give a proof in the case where $A$ is finitely graded. 

\begin{proposition}\label{finitely graded projective dimension lemma}
For a graded $A$-module $M$, we have $\pd M = \grpd M$.  
\end{proposition}

\begin{proof}
The inequality $\pd M \leq \grpd M$ follows from the fact that 
if we forget the grading from a graded projective $A$-module, then it become a projective $A$-module.

We prove $\pd M \geq \grpd M$. 
We may assume  $n:= \pd M < \infty$. 
Take a projective resolution 
$ \cdots \to P^{-1} \xrightarrow{\partial^{-1}} P^{0} \to M$ in $\GrMod A$.
It is enough to show that $K : = \Ker \partial^{-(n-1)}$ is a  graded projective $A$-module. 
Since $K$ is a projective $A$-module, 
the graded $\Lambda$-module 
$L := K \otimes_{A} \Lambda$ become a projective $\Lambda$-module 
if we forget the grading.  
Hence $L$ is  projective as a graded $\Lambda$-module. 
The canonical surjection $\mathsf{cs} : K \to L$ has a section $\mathsf{sc} : L \to K $ in $\GrMod \Lambda$, 
which  extends to a graded $A$-module homomorphism $\widehat{\mathsf{sc}}: L \otimes_{\Lambda} A \to K$. 
In the same way of the proof of Lemma \ref{projective decomposition lemma},  we can check that $\widehat{\mathsf{sc}}$ is an isomorphism.
\end{proof}

\subsection{Graded injective dimension and ungraded injective dimension of graded $A$-modules}

The following Lemma can be easily checked and  is left to the readers. 

\begin{lemma}\label{basic injective lemma}
\begin{enumerate}[(1)]
\item Let $I$ be a graded injective $A$-module. 
Then, for all $i \in \ZZ$,  
the $\Lambda$-module $\grHom_{A}(\Lambda, I)_{i}$ is an injective $\Lambda$-module. 
Moreover, 
we have $\grHom_{A}(\Lambda, I) = \Hom_{A}(\Lambda, I)$ as ungraded $\Lambda$-modules.

\item Let $J$ be an injective $\Lambda$-module. 
We regard $J$ as a graded $\Lambda$-module concentrated in degree $0$. 
Then the graded $A$-module $\grHom_{\Lambda}(A, J)$ is a graded injective $A$-module. 
Moreover we have $\grHom_{\Lambda}(A,J) =\Hom_{\Lambda}(A, J)$ as ungraded $A$-modules 
and it is an ungraded injective $A$-module. 
\end{enumerate}
\end{lemma}

For an integer $i \in \ZZ$, 
we denote  by
$\frki_{i}: \GrInj A \to \Inj \Lambda$ the functor $\frki_{i}I := \grHom_{A}(\Lambda, I)_{i}$. 
Then, we obtain a graded injective $A$-module  
$\frks_{i}I := \grHom_{\Lambda}(A,  \frki_{i}I)(-i)$.  

\begin{lemma}\label{decomposition of graded injective modules} 
Let $I$ be an object of $\GrInj A$.   
Then,  we have the following isomorphism of graded $A$-modules 
\[
 I  \cong \bigoplus_{ i\in \ZZ} \frks_{i} I.
 \]
\end{lemma}  
 
\begin{proof}
First observe that   the direct sum $\tilde{I} := \bigoplus_{i\in \ZZ} \frks_{i} I$ is finite in each degree. 
More precisely, 
$\tilde{I}_{j} = \bigoplus_{i \in \ZZ} (\frks_{i}I)_{j} = \bigoplus_{i=j }^{j+ \ell} (\frks_{i}I)_{j}$. 
It follows that $\tilde{I}$ is the direct product of $\{\frks_{i} I\}_{i\in \ZZ}$ in the category $\GrMod A$. 
Therefore,  $\tilde{I}$ is an injective object of $\GrMod A$

It can be checked that both $I$ and $\bigoplus_{i\in \ZZ} \frks_{i} I$ 
contain $\grHom_{A} (\Lambda, I) = \bigoplus_{i \in \ZZ} \frki_{i} I$ 
as an essential submodule. 
Hence by uniqueness of injective hull, we conclude the desired isomorphism. 
\end{proof}

The following Lemma is an injective version of Lemma \ref{projective Hom decomposition corollary}. 
The proof is left to the readers.

\begin{corollary}\label{injective Hom decomposition corollary}
We have the following isomorphism for $I \in \GrInj A$ and $M \in \GrMod A$. 
\[
\Hom_{\GrMod A}(M, I ) \cong \prod_{i\in \ZZ} \Hom_{\Lambda}(M_{i} , \frki_{i}I). 
\]
\end{corollary}

For a finitely graded module $M$, 
its graded injective  dimension and (ungraded) injective dimension coincide. 

\begin{proposition}\label{finitely graded injective dimension lemma}
For a finitely graded $A$-module $M$, we have  $\injdim M = \grinjdim M$. 
\end{proposition}

\begin{proof}
We claim that 
an injective hull $I$ of $M$ in $\GrMod A$ is finitely graded and is an injective hull in $\Mod A$. 
Indeed,  it can be checked that $I$  is finitely graded and  $I = \bigoplus_{i \in \ZZ} \frks_{i}I$ is finite sum. 
Therefore it is  injective as an ungraded $A$-module. 
It can be checked that a graded essential submodule $N \subset L$ of a graded module $L$  
is an essential (ungraded) submodule (and vice versa) \cite[I.3.3.13]{NV:Graded and Filtered}. 
In particular,  $M$ is an essential $A$-submodule of $I$. 
Thus,  we conclude that $I$ is an injective hull of $M$ in $\Mod A$.

Using the claim, we can easily  check  $\injdim M \leq \grinjdim M$. 

Assume that $n := \injdim M < \infty$. 
Take a minimal injective resolution 
$ 0 \to M \to I^{0} \xrightarrow{ \partial^{0}} I^{1} \to \cdots $ in $\GrMod A$. 
Then $K =\Coker \partial^{n-2}$ is finitely graded and  injective  as an ungraded $A$-module. 
It follows from  the claim that  $K$ is a graded injective $A$-module. 
Hence we conclude that $\injdim M \geq \grinjdim M$. 
\end{proof}

We point out the following immediate consequence. 

\begin{corollary}\label{finitely graded IG-lemma}
We have the following equation
\[
\injdim A_{A} = \grinjdim A_{A}.
\]
\end{corollary}

The graded global dimension and the (ungraded) global dimension coincide. 
Moreover this can be measured by  the category of finitely graded modules. 

\begin{proposition}[{cf. \cite[I.7.8]{NV:Graded and Filtered}}]\label{global dimension lemma} 
The following equations hold: 
\[
\begin{split}
\gldim A  & = \grgldim A \\
                &= \sup \{ \grpd M \mid M \textup{ a finitely graded $A$-module.}\}\\
                &= \sup \{ \grinjdim M \mid M \textup{ a finitely graded $A$-module.}\}\\
                &= \sup \{ \pd M \mid M  \textup{ an $A$-module such that $MA_{\geq 1} = 0$.}\}\\
                             &= \sup \{ \injdim M \mid M  \textup{ an $A$-module such that $MA_{\geq 1} = 0$.}\}
                 \end{split}
\]
\end{proposition}

\begin{proof}
For simplicity, we denote by  $v_{p}$  the $p$-th value of the above equation. 
For example, $v_{1} := \gldim A, \ v_{2} := \grgldim A$.  
The equation $ v_{1} \geq v_{2}$ follows from 
Proposition \ref{finitely graded projective dimension lemma}. 
It is clear that $v_{2} \geq v_{3}$. 
Since an $A$-module $M$ such that $MA_{\geq 1} = 0$ can be regarded as a graded $A$-module 
concentrated at $0$-th degree, 
we can see that $v_{3} \geq v_{5}$ by Proposition \ref{finitely graded projective dimension lemma}.  
We prove the inequality $v_{5} \geq v_{1}$. 
Since an  $A$-module $M$ has a filtration  $M_{i} := MA_{\geq i}$ for $i = 1, \dots ,\ell$, 
whose graded quotients $N_{i} =M_{i}/M_{i+1}$ satisfy $N_{i} A_{\geq 1} =0$, 
we conclude that $\pd M \leq  v_{5}$. 
In the same way, we can prove $v_{2} \geq v_{4} \geq v_{6} \geq v_{1}$.  
\end{proof}

\subsection{Quasi-Veronese algebras and Beilnson algebras for finitely graded algebras}\label{quasi-Veronese}

This Section \ref{quasi-Veronese} does  not relate  to the main subject of this paper in a strict sense. 
We explain an importance of trivial extension algebras 
by showing that 
every finitely graded algebra $A$ is graded Morita equivalent to 
the  trivial extension algebra  $\nabla A \oplus \Delta A$ equipped with the standard grading.

Let $A$ be a finitely graded algebra. 
We fix a natural number $\ell$ such that $A_{i} = 0$ for $ i \geq \ell +1$. 
(It is not necessary to assume that $A_{\ell} \neq 0$.)  
In this situation, we define \textit{the Beilinson algebra} $\nabla A$ of $A$ 
(which  rigorously should be  called the Beilinson algebra of the pair $(A, \ell)$) 
and its bimodule $\Delta A$ to be 
\[
\nabla A: = 
\begin{pmatrix} 
A_{0} & A_{1} & \cdots & A_{\ell -1} \\
0         & A_{0} & \cdots  & A_{\ell -2} \\
\vdots & \vdots     &        & \vdots \\
0         & 0   & \cdots  & A_{0}
\end{pmatrix}, \ \ \ 
\Delta A: = 
\begin{pmatrix} 
A_{\ell} & 0 & \cdots & 0 \\
A_{\ell- 1} & A_{\ell} & \cdots  &0 \\
\vdots & \vdots     &        & \vdots \\
A_{1} & A_{2}   & \cdots  & A_{\ell}
\end{pmatrix} 
\]
where the algebra structure and the bimodule structure are 
given by the  matrix  multiplications. 
Then, 
the trivial extension algebra 
 $\nabla A \oplus \Delta A$ 
with the grading $\deg \nabla A = 0, \deg  \Delta A = 1$ 
is  nothing but the $\ell$-th quasi-Veronese algebra $A^{[\ell]}$ of $A$ 
introduced by Mori \cite[Definition 3.10]{Mori B-construction}. 
\[
A^{[\ell]} = \nabla A \oplus \Delta A. 
\]

\begin{remark}
In \cite{Mori B-construction}, the multiplication of the quasi-Veronese algebra 
is defined by so called the opposite of the matrix  multiplication, 
which is different from our definition. 
However, this difference occurs from the notational difference. 
\end{remark}

By \cite[Lemma 3.12]{Mori B-construction} 
$A$ and $A^{[\ell]} $ are graded Morita equivalent to each other. 
More precisely, 
the functor $\sf{qv}$ below gives a $\kk$-linear equivalence. 
\[
\begin{split}
 & \mathsf{qv}: \GrMod A \xrightarrow{ \ \simeq  \ } \GrMod A^{ [ \ell ]}, \\
& \mathsf{qv}(M) := \bigoplus_{i \in \ZZ} \mathsf{qv}(M)_{i}, \ \ \ \ 
\mathsf{qv}(M)_{i} =  M_{i\ell} \oplus M_{i \ell +1 } \oplus \cdots \oplus M_{( i+ 1)\ell -1 }
\end{split}
\]
We note that $\mathsf{qv}$ does not commute with the degree shifts $(1)$ of each categories. 
However, under this equivalence 
the degree shift $(\ell)$ by $\ell$ of $\GrMod A$ corresponds to 
the degree shift $(1)$ by $1$  of $\GrMod A^{[\ell]}$.

We give a quick explanation. 
We set $V := A \oplus A(-1) \oplus \cdots \oplus A(- (\ell -1))$. 
Then the set $\{ V(i \ell ) \mid i \in \ZZ \}$ is 
a set of finitely generated projective generators 
of the abelian category $\GrMod A$. 
Therefore, in the same way of  Morita theory, 
we can construct a $\kk$-linear equivalence 
$\mathsf{qv} : \GrMod A \to \GrMod E$ 
by using $V$ and $(\ell)$, 
where 
$E : = \bigoplus_{i \in \ZZ} \Hom_{\GrMod A} (V, V(i  \ell ) ) $
 is the graded algebra whose multiplication is given by 
 the twisted composition. 
Now by $\Hom_{\GrMod A}(A(-j), A(-i)) \cong A_{j -i}$, 
we can check that $E = A^{[\ell]}$.  
For details, we refer \cite{Mori B-construction}. 

We  
give a summary of the  above consideration and  collect  consequences. 

\begin{lemma}\label{Quasi-Veronese lemma}
Let $A= \bigoplus_{i =0}^{\ell} A_{i}$ be a finitely graded algebra. 
Then, 
$A$ 
is graded Morita equivalent to 
the trivial extension algebra $A^{[\ell]} = \nabla A \oplus \Delta A$.

Moreover the following assertions hold. 
\begin{enumerate}[(1)]
\item $A$ is of finite global dimension if and only if so is $A^{[\ell]}$. 

\item $A$ is Iwanaga-Gorenstein if and only if so is $A^{[\ell]}$. 

\end{enumerate}
\end{lemma}
The definition of Iwanaga-Gorenstein algebras  will be recalled in Section \ref{A criterion of Iwanaga-Gorensteinness}.

\section{Homological dimensions of (unbounded) complexes (after Avramov-Foxby)}\label{Homological dimension}

In this Section \ref{Homological dimension}, 
we recall the definition of homological  dimensions of (unbounded) complexes 
introduced by Avramov-Foxby \cite{Avramov-Foxby}. 

If the reader is only  interested in suitably bounded derived categories, 
it is sufficient to use complexes of projective (resp. injective) modules 
bounded above (resp. below) in place of  DG-projective (resp. DG-injective) complexes whose definition is recalled in 
Definition \ref{DG-projective definition}. 
The statements concerning on DG-projective (resp. DG-injective) complexes 
can be proved  easily for complexes of projective (resp. injective) modules bounded above (resp. below).

For an additive category $\cA$, we denote by $\sfC(\cA)$ and $\sfK(\cA)$  
the category of cochain complexes and cochain morphisms 
and its homotopy category respectively. 
For complexes $X, Y \in \sfC(\cA)$, 
we denote by $\Hom_{\cA}^{\bullet}(X, Y)$ the $\Hom$-complex. 
For an abelian category $\cA$, we denote by $\sfD(\cA)$ the derived category of $\cA$. 

The shift functor of a triangulated category  is denoted by $[1]$.  
The cone of a morphism $f$ in a triangulated category is denoted by $\cone f$.

\subsection{Unbounded derived categories}

We recall a projective resolution and  an injective resolution of unbounded complexes.  
For the details we refer \cite{Keller:ddc}, \cite{Keller:ICM}, \cite{Positselski}. 

\begin{definition}\label{DG-projective definition}
Let $\Lambda$ be an algebra. 
\begin{enumerate}[(1)] 
\item An object $M \in  \sfC(\Mod \Lambda)$ is called \textit{homotopically projective} 
if the complex 
$\Hom_{\Lambda}^{\bullet}(M,N) $ is acyclic 
for an acyclic complex $N \in \sfC(\Mod \Lambda)$. 

\item An object $M \in  \sfC(\Mod \Lambda)$ is called \textit{homotopically injective} 
if the complex 
$\Hom_{\Lambda}^{\bullet}(N, M) $ is acyclic 
for an acyclic complex $N \in \sfC(\Mod \Lambda)$. 

\item 
An object $P \in  \sfC(\Mod \Lambda)$ is called \textit{DG-projective} 
if  it belongs to $\sfC(\Proj \Lambda)$ and is homotopically projective. 

\item 
An object $I \in  \sfC(\Mod \Lambda)$ is called \textit{DG-injective} 
if  it belongs to $\sfC(\Inj \Lambda)$ and is homotopically injective. 
\end{enumerate}
\end{definition}

Let $M$ be an object of $\sfC(\Mod \Lambda)$. 
A quasi-isomorphism $f: P \to M$ with $P$ DG-projective 
is called a \textit{projective resolution} of $M$. 
We often say that $P$ is a DG-projective resolution of $M$ by suppressing a quasi-isomorphism $f: P \to M$. 
A quasi-isomorphism $f: M \to I$ with $P$ DG-injective  
is called an \textit{injective resolution} of $M$. 
We often say that $I$ is a DG-injective resolution of $M$ by suppressing a quasi-isomorphism $f: M \to I$.

It is known that every object of $M \in \sfC(\Mod \Lambda)$ has a projective resolution and an injective resolution \cite[1.4, 1.5]{Positselski}. 
By \cite[8.1]{Positselski}, DG-projective complexes and DG-injective complexes coincide with 
cofibrant DG-modules and fibrant DG-modules in \cite{Keller:ICM}. 
Existence of projective resolutions and injective resolutions are noted in \cite[Proposition 3.1]{Keller:ICM}.

Let $M \in \sfC(\Mod \Lambda)$ and $D$ a $\Lambda$-$\Lambda$-bimodule. 
We recall that the derived tensor product $M \lotimes_{\Lambda} D$ and the derived $\Hom$-space 
$\RHom_{\Lambda}(D, M)$ are defined by using a projective resolution $P \xrightarrow{\sim} M$ 
and an injective resolution $M \xrightarrow{\sim} I$ 
respectively. Namely, $M \lotimes_{\Lambda}D := P \otimes_{\Lambda}D$ 
and $\RHom_{\Lambda}(D, M) := \Hom_{\Lambda}^{\bullet}(D, I)$.

We denote by $\sfK_{\mathsf{proj}}(\Proj \Lambda)$ \ (resp. $\sfK_{\mathsf{inj}}(\Inj \Lambda)$ )  
the full subcategory of  $\sfK(\Mod \Lambda)$ consisting of the homotopy classes of DG-projective (resp. DG-injective)  complexes.  

Then the composite functors below are equivalences of triangulated categories 
\[
\begin{split}
&\sfK_{\mathsf{proj}}(\Proj \Lambda) \hookrightarrow \sfK(\Mod \Lambda) \xrightarrow{ \ \ \mathsf{pr} \ \ } 
\sfD(\Mod \Lambda)\\
&\sfK_{\mathsf{inj}}(\Inj \Lambda) \hookrightarrow \sfK(\Mod \Lambda) \xrightarrow{ \ \ \mathsf{pr} \ \ } 
\sfD(\Mod \Lambda)
\end{split}
\]
where $\mathsf{pr}$ is a canonical projection. 
Moreover, the functor $\mathsf{pr}$ induces isomorphisms below 
\[
\begin{split}
&\Hom_{\sfK( \Mod \Lambda) }(P, M) \cong \Hom_{\sfD(\Mod \Lambda)}(P,M)\\
&\Hom_{\sfK( \Mod \Lambda) }(M, I) \cong \Hom_{\sfD(\Mod \Lambda)}(M,I)
\end{split}
\]
for $P \in \sfK_{\mathsf{proj}}(\Proj \Lambda)$, $I \in \sfK_{\mathsf{inj}}(\Inj \Lambda)$ and $M \in \sfK(\Mod \Lambda)$, 
where we suppress $\mathsf{pr}$ in the right hand sides.

\subsection{Projective dimension and injective dimension of  unbounded complexes}

We recall the definition of 
projective dimension and injective dimension for unbounded complexes 
introduced by Avramov-Foxby \cite{Avramov-Foxby}. 

\begin{definition}[{\cite[Definition 2.1.P and 2.1.I and Theorem 2.4.P and 2.4.I]{Avramov-Foxby}}]\label{definition of homological dimension}
\begin{enumerate}[(1)]

\item  
An object $M$ of $\sfD(\Mod \Lambda)$ is said to have 
\textit{projective dimension} at most $n$ and is denoted as below 
if it has a projective resolution $P$ 
such that $P^{m} = 0$ for $m < -n$. 
\[
\pd_{\Lambda} M \leq n. 
\]
We write $\pd_{\Lambda} M = n $ if $\pd_{\Lambda} M \leq n$ holds
but $\pd_{\Lambda} M \leq n-1$ does not.

\item 
An object $M$ of $\sfD(\Mod \Lambda)$ is said to have 
\textit{injective dimension} at most $n$ 
and is denoted as below  
if it has an injective resolution $I$ 
such that $I^{m} = 0$ for $m > n$. 
\[
\injdim_{\Lambda} M \leq n. 
\] 
We write $\injdim_{\Lambda} M = n $ if $\injdim_{\Lambda} M \leq n$ holds
but $\injdim_{\Lambda} M \leq n-1$ does not.

\end{enumerate}
\end{definition}

From now on,
if there is no danger of confusion, 
we usually write the projective dimension and the injective dimension as  $\pd$ and $\id$.

If $\pd M \leq n$ (resp. $\injdim M < n$) 
for $n \in \ZZ$, then we write $\pd M = - \infty$ (resp. $\injdim M = - \infty$). 
However, $\pd  M = -\infty \Leftrightarrow \injdim M = -\infty \Leftrightarrow M = 0$ 
in $\sfD(\Mod \Lambda)$ (\cite[2.3.P and I]{Avramov-Foxby}).

\begin{theorem}[{Avramov-Foxby \cite[Theorem 2.4.P and 2.4.I]{Avramov-Foxby}}]\label{Avramov-Foxby:theorem 2.4}
Let  $M \in \sfD(\Mod \Lambda)$ be an object and $n \in \ZZ$. 
Then the following assertions hold.
\begin{enumerate}[(1)]
 \item 
The following conditions  are equivalent. 
\begin{enumerate}
\item $\pd M \leq n$. 

\item $\tuH^{m}(M) = 0$ for $m < -n$ 
and there exists a projective resolution 
 $P$ of $M$  
such that $\Coker d_{P}^{-n-1}$ is a projective $\Lambda$-module 
 where $d_{P}^{-n-1}$ is the $(-n-1)$-th differential of $P$. 

\item 
$\tuH^{m}(M) = 0$ for $m < -n$ 
and for any projective resolution $P$ of $M$, $\Cok \partial_{P}^{-n-1}$ is a projective $\Lambda$-module.
\end{enumerate}

\item 
The following conditions are equivalent. 
\begin{enumerate}
\item $\injdim M \leq n$. 

\item 
$\tuH^{m}(M) = 0$ for $m > n$ and 
there exists an injective resolution 
 $I$ of $M$  such that $\Ker d_{I}^{n}$ is an injective $\Lambda$-module 
 where $d_{I}^{n}$ is the $n$-th differential of $I$. 

\item 
$\tuH^{m}(M) = 0$ for $m > n$ and 
for any injective resolution $I$ of $M$, $\Ker \partial_{I}^{n}$ is an injective $\Lambda$-module.
\end{enumerate}
\end{enumerate}
\end{theorem}

Later we will use the following results several times. 

\begin{theorem}[{\cite[2.3.F.(5) and Theorem 4.1]{Avramov-Foxby}}]\label{Avramov-Foxby:inequality}
Let $M$ be a complex of $\Lambda$-modules and $D$ a complex of $\Lambda$-$\Lambda$-bimodules. 
The we have 
\[
\pd M \lotimes_{\Lambda} D \leq \pd M + \pd D_{\Lambda}, \ \ \
\injdim \RHom_{\Lambda}(D, M) \leq \pd {}_{\Lambda} D + \injdim M. 
\]
\end{theorem}

We point out the following basic properties. 
The proofs are left to the readers. 

\begin{lemma}\label{homological dimension lemma}
\begin{enumerate}[(1)]
\item 
Let $L \to M \to N \to $ be an exact triangle in $\sfD(\Mod \Lambda)$. 
Then, 
\[
\pd M \leq \sup\{\pd L, \pd N \}, \ \ \injdim M \leq \sup\{ \injdim L, \injdim N\}. 
\]

\item 
Let $M \in \sfD(\Mod \Lambda)$ and $n \in \ZZ$. 
Then we have 
\[
\pd( M[n] ) = \pd M + n, \ \ \injdim( M[n] ) =\injdim N - n.
\]
\end{enumerate}
\end{lemma}

For a graded algebra $A$, 
we define the graded version of the above notions 
in a  similar way of Definition \ref{definition of homological dimension} 
by using $\GrMod A,  \GrProj A$ and $\GrInj A$  in place of $\Mod A , \Proj A$ and $\Inj A$.  
The projective dimension and the  injective dimension of an object $M \in \sfC(\GrMod A)$ 
are denoted by $\grpd_{A} M$ and $\grinjdim_{A} M$.

The following is a complex version of Proposition \ref{finitely graded projective dimension lemma}. 

\begin{lemma}\label{finitely graded projective dimension complex lemma}
Let $A$ be a finitely graded algebra 
and $M \in \sfD(\GrMod A)$. 
Then, we have 
\[
\grpd M = \pd M.  
\]
\end{lemma}

\begin{proof}
We have only  to prove that 
if $P \in \sfC(\GrProj A)$ is DG-projective, 
then the underlying complex $\sfU(P)$ of (ungraded) $A$-modules  is DG-projective in $\sfC(\Mod  A)$. 
After that, thanks to Theorem \ref{Avramov-Foxby:theorem 2.4}, 
the rest of the proof is the same with that of the Proposition \ref{finitely graded projective dimension lemma}.

It is enough to  show that $\sfU(P)$ is homotopically projective. 
For this we recall the  notion of property (P) from \cite{Keller:ddc}.  
An object $P$ of $\sfC(\GrProj A)$ 
is said to have property (P) if it has an increasing  filtration $0 = P_{(0)} \subset P_{(1)} \subset \cdots  $ 
such that each graded quotient $P_{(i)}/P_{(i -1)}$ is an object of $\sfC(\GrProj A)$ with $0$ differential 
and that $\bigcup_{i \geq 0} P_{(i)} = P$. 
By \cite[3.1]{Keller:ddc}  an object $Q \in \sfC(\GrProj A)$ having property (P) is DG-projective in particular it is homotopically projective. 
We claim that every DG-projective complex $P$ is a direct summand of a complex $Q \in \sfC(\GrProj  A)$ 
having property (P).  
Indeed, by \cite[3.1]{Keller:ddc}, 
there exists a quasi-isomorphism $f: Q \stackrel{\sim}{\twoheadrightarrow} P$ 
from a complex $Q \in \sfC(\GrProj A)$ having property (P) 
such that each component $f^{n}: Q^{n} \to P^{n}$ is surjective. 
Then, since $P$ is cofibrant, $f$ has a section  by \cite[3.2]{Keller:ICM}. 

If $P \in \sfC(\GrProj A)$ has property (P), then so is $\sfU(P) \in \sfC(\Proj A)$.  
Hence it is homotopically projective.  
Let $P$ be a DG-projective object of $\sfC(\GrProj A)$. 
Then by the claim,  $\sfU(P)$ is a direct summand of a homotopically projective object of $\sfC(\Mod A)$ 
and hence is homotopically  projective. 
\end{proof}

The following is a complex version of Proposition \ref{finitely graded injective dimension lemma}.

\begin{lemma}\label{finitely graded injective dimension complex lemma}
Let $A$ be a finitely graded algebra 
and $M \in \sfD^{+}(\GrMod A)$ such that $M_{i} = 0$ for $ |i| \gg 0$. 
Then, we have 
\[
\grinjdim M = \injdim M.  
\]
\end{lemma} 

\begin{proof}
Under the assumption   $M$ has an injective resolution $I$ 
such that each term $I^{n}$ is finitely graded. 
Using Theorem \ref{Avramov-Foxby:theorem 2.4}.(b), 
we can prove the assertion as in the same of Proposition \ref{finitely graded injective dimension lemma}. 
\end{proof}

\section{Projective dimension formula} 
\label{A criterion of perfectness}

In this Section \ref{A criterion of perfectness}, 
$\Lambda$ denotes an algebra, $C$ denotes a bimodule over it 
and $A = \Lambda \oplus C$ is the trivial extension algebra. 
We establish the projective dimension formula for  
an object $M \in \sfD^{\mrb}(\GrMod A)$ such that $M_{i} = 0$ for $i \neq 0,1$. 
The key tool is a  decomposition of 
a complex $P \in \sfC(\GrProj A)$ of graded projective $A$-modules  
according to degree of the generators  
which was introduced in \cite{Orlov}.

\subsection{Decomposition of complexes of graded  projective $A$-modules}\label{Decomposition of a graded projective complexes}

\subsubsection{Decomposition of graded projective $A$-modules}\label{Decomposition of graded projective $A$-module}


Recall that 
 $\frkp_{i} P = (P \otimes_{A} \Lambda)_{i}, \ \frkt_{i}P := (\frkp_{i}P ) \otimes_{\Lambda} A (-i)$ 
for $P \in \GrProj A$.   
Since now $A = \Lambda \oplus C$ is a trivial extension algebra,  
we have an isomorphism $\frkt_{i} P = (\frkp_{i} P)(-i) \oplus (\frkp_{i} P )\otimes_{\Lambda} C (-i-1)$ 
of $\Lambda$-modules. 
Therefore, by Lemma \ref{projective decomposition lemma}, we have the following isomorphisms. 

\begin{equation}\label{decomposition of graded projective modules}
\begin{split}
 P  & \cong \bigoplus_{ i\in \ZZ} \frkt_{i} P  \ (\textup{as graded $A$-modules}), \\
 P_{i} & \cong (\frkt_{i-1} P)_{i}  \oplus  (\frkt_{i} P)_{i} \cong (\frkp_{i-1} P \otimes_{\Lambda} C) \oplus \frkp_{i}P 
\ ( \textup{as $\Lambda$-modules}). 
\end{split} 
\end{equation}

We discuss  decomposition of a morphism $f: P \to P'$ of graded projective $A$-modules. 
 The case where 
 the generating spaces of  $P$ and $P'$ consist of single degree  
can be easily deduced from Corollary  \ref{projective Hom decomposition corollary}. 

\begin{lemma}\label{basics of projectives}
Let $P, P'$ be graded projective $A$-modules 
such that $P = \frkt_{i} P, \ P' = \frkt_{j} P'$ 
for some $i ,j  \in \ZZ$. 
Then, we have the following isomorphism  
\[
\Hom_{\GrMod A}(P,P') 
= \grHom_{A}(P,P')_{0}
\cong \begin{cases}
0 &   j \neq i, i -1, \\
 \Hom_{\Lambda}(P_{i}, P'_{i} ) &  j = i,  \\
 \Hom_{\Lambda}(P_{i}, P'_{i - 1} \otimes_{\Lambda} C) &  j = i -1.  
 \end{cases}
 \]
\end{lemma}

Let $f: P \to P' $ be a morphism in $\GrMod A$  with $P, P' \in \GrProj A$. 
We denote by $f_{ij}: \frkt_{j} P \to \frkt_{i} P'$ 
the component of $f$ with respect to the decompositions (\ref{decomposition of graded projective modules}) 
of $P$ and $P'$, that is, 
$f_{ij}$ is defined by the following composition
\[
f_{ij}: \frkt_{j} P  \xrightarrow{\ \ \mathsf{in}_{j} \ \ } P \xrightarrow{f}  P' \xrightarrow{\ \ \mathsf{pr}_{i} \ \ } \frkt_{i} P'
\]
where $\mathsf{in}_{j}$ is a canonical inclusion 
and $\mathsf{pr}_{i}$ is a canonical projection. 

Note that  if $j-i \neq 0,1$, then $f_{ij} = 0$ 
and that  $\frkp_{i}( f )= (f_{ii})_{i} : \frkp_{i} P = (\frkt_{i} P)_{i} \to (\frkt_{i} P' )_{i} = \frkp_{i} P'$. 


We set $\frkq_{i}(f)  := (f_{i-1,i})_{i} : \frkp_{i} P \to \frkp_{i-1} P' \otimes_{\Lambda} C$.
Then the degree $i$-part $f_{i} : P_{i}  \to P'_{i} $  is isomorphic to 
\[
\begin{pmatrix} 
\frkp_{i-1} (f )\otimes C & \frkq_{i}(f)  \\
0 & \frkp_{i}( f ) 
\end{pmatrix}
: (\frkp_{i-1} P \otimes_{\Lambda} C) \oplus \frkp_{i}P \to 
(\frkp_{i-1} P' \otimes_{\Lambda} C) \oplus \frkp_{i}P'
\] 
under the isomorphisms (\ref{decomposition of graded projective modules}),  
where we use the standard matrix  multiplication 
\[
\begin{pmatrix} 
\frkp_{i-1} (f )\otimes C & \frkq_{i}(f)  \\
0 & \frkp_{i}( f ) 
\end{pmatrix}
\begin{pmatrix}
p \\ q 
\end{pmatrix} 
= \begin{pmatrix}
(\frkp_{i-1}(f)  \otimes C)(p)  + \frkq_{i}(f) (q) \\  \frkp_{i}(f)(q) 
\end{pmatrix} 
\]
for $p \in \frkp_{i-1}P \otimes C, q \in \frkp_{i} P$. 

Let $i$ be an integer. 
We set $\frkt_{<i}P := \bigoplus_{ j< i} \frkt_{j} P$ and $\frkt_{\geq i} P := \bigoplus_{j \geq i} \frkt_{j} P$. 
For a graded $A$-module homomorphism 
$f: P \to P'$, 
we denote by 
$\frkt_{< i}(f) : \frkt_{< i} P \to \frkt_{< i} P' $ and 
$\frkt_{\geq i} (f) :\frkt_{\geq i} P \to \frkt_{\geq i} P'$ 
the induced morphisms. 
Then the canonical inclusion 
$\frkt_{<i} P \to P $ and 
the canonical projection $P \to\frkt_{\geq i} P$ 
fit into the following commutative diagram  
\[
\begin{xymatrix}{ 
0 \ar[r] &\frkt_{<i} P \ar[r]  \ar[d]_{\frkt_{<i } (f) }   & 
P \ar[d]^{f} \ar[r] 
& \frkt_{\geq i} P \ar[d]^{\frkt_{\geq i} (f) }  \ar[r] & 0\\  
0 \ar[r] &\frkt_{<i} P' \ar[r]  & P'  \ar[r] & \frkt_{\geq i}P' \ar[r] & 0 
}\end{xymatrix}
\]
where the both rows are split  exact.

\subsubsection{Decomposition of  complexes of graded projective modules}\label{A decomposition of  complexes of graded projective modules}

By abuse of notations, 
we denote the functors 
$\frkp_{i}: \sfC(\GrProj A ) \to \sfC(\Proj \Lambda)$ and 
$\frkp_{i}: \sfK(\GrProj A) \to \sfK(\Proj \Lambda)$ 
induced  from the functor $\frkp_{i}: \GrProj A \to \Proj \Lambda$ 
by the same symbols.

Let 
$P= ( \bigoplus_{n \in \ZZ} P^{n}, \{d^{n}\}_{n \in \ZZ} )  \in \sfC(\GrProj A)$ 
be a complex of graded projective $A$-modules. 
For simplicity, we set $\partial_{i}^{n} : = \frkp_{i}( d^{n}), \sfq_{i}^{n}:= \frkq_{i}(d^{n})$.  

Looking at graded degree $i$-part, we have the following  isomorphism 
\begin{equation}\label{projective complex isomorphism 1}
P^{n}_{i} = (\frkt_{i-1}P^{n})_{i} \oplus (\frkt_{i} P^{n})_{i} \cong 
(\frkp_{i-1}P^{n} \otimes_{\Lambda} C) \oplus  \frkp_{i} P^{n}
\end{equation}
of $\Lambda$-modules. 
Via this isomorphism, the differential $d^{n}:P^{n} \to P^{n+1}$ is decomposed into 
\[
\begin{pmatrix} 
\partial_{i-1}^{n} \otimes C &  \sfq_{i}^{n} \\ 
0 & \partial_{i}^{n} 
\end{pmatrix}
: (\frkp_{i-1}P^{n} \otimes_{\Lambda} C) \oplus  \frkp_{i} P^{n} \to 
(\frkp_{i-1}P^{n+1} \otimes_{\Lambda} C) \oplus  \frkp_{i} P^{n+1}.
\]
Now it is clear that the canonical morphisms $\frkp_{i-1} P \otimes_{\Lambda} C \to P_{i}, \ 
P_{i} \to \frkp_{i} P$ commute with the differentials.   
Hence we obtain an exact sequence in $\sfC(\Mod \Lambda)$
\[
0\to
\frkp_{i-1}P \otimes_{\Lambda} C \to P_{i} \to \frkp_{i} P \to 0 
\]
which splits when we forget the differentials. 
Moreover,  
the complex $P_{i}$ is obtained as the cone of 
the morphism $\sfq_{i}[-1]:  \frkp_{i} P [-1] \to  \frkp_{i-1} P \otimes C$ 
which is  given by the collection $\{\sfq_{i}^{ n }\}_{n\in \ZZ}$. 
This observation yields the following lemma.

\begin{lemma}\label{proj:description of cohomology}
We have an exact triangle in $\sfK(\Mod \Lambda)$ and hence in $\sfD(\Mod \Lambda)$   
\[
\frkp_{i-1}P \otimes_{\Lambda} C \to P_{i} \to \frkp_{i} P 
\xrightarrow{\sfq_{i}} \frkp_{i-1} P \otimes C[1].  
\]

In particular, $\tuH(P)_{i} = \tuH(P_{i}) =0$ if and only if the morphism 
$\sfq_{i}: \frkp_{i}P \to \frkp_{i-1} P \otimes_{\Lambda} C[1]$ is  a quasi-isomorphism. 
\end{lemma}

For an integer $i$,  
we set
 $\frkt_{< i} P := ( \bigoplus_{n \in \ZZ} \frkt_{ < i} P^{n}, \{ \frkt_{< i} (d^{n})\}_{n \in \ZZ} ) $
and 
 $\frkt_{\geq i} P := ( \bigoplus_{n \in \ZZ} \frkt_{\geq i} P^{n}, \{\frkt_{\geq i} (d^{n})\}_{n \in \ZZ} )$. 
 Then the canonical inclusion 
$\frkt_{<i} P \to P $ and 
the canonical projection $P \to\frkt_{\geq i} P$ 
form an exact sequence of $\sfC(\grproj A)$
\[
0\to \frkt_{< i} P \to P \to \frkt_{\geq i} P \to 0, 
\]
which gives an exact triangle in $\sfK(\GrProj A)$
\begin{equation}\label{ad hoc ref 3/20}
 \frkt_{< i} P \to P \to \frkt_{\geq i} P  \to \frkt_{< i} P[1]. 
\end{equation}

We discuss  properties of the functor $\frkp_{i}: \sfK(\GrProj A) \to \sfK(\Proj \Lambda)$ a bit more. 

\begin{lemma}\label{projectively cofibrant lemma}

\begin{enumerate}[(1)]
\item If $P$ is DG-projective, then so is  $\frkp_{i} P$ for $i \in \ZZ$.

\item  
Assume that $\frkp_{i}P $ is DG-projective for $i \in \ZZ$ and that 
$\frkp_{i} P = 0$  for $i \ll 0$,  
then so is $P$. 
\end{enumerate}
\end{lemma}

We use the following description of $\Hom$-complex. 
 We denote by $\Hom^{\#}$ 
the underlying cohomologically graded $\kk$-modules of  $\Hom$-complexes 
$\Hom^{\bullet}$. 

Let $M \in \sfC(\GrMod A)$. The isomorphisms of Corollary \ref{projective Hom decomposition corollary} induces 
 an isomorphism of cohomologically graded $\kk$-modules. 
 \[ \Phi: \Hom_{\GrMod A}^{\#}(P, M) \xrightarrow{\cong} \prod_{i \in \ZZ} \Hom_{\Lambda}^{\#}(\frkp_{i}P, M_{i}). 
 \]
However, this is not  compatible with the canonical differentials.   
Via this isomorphism, the differential $\partial_{\Hom}$ of $\Hom_{\GrMod A}^{\bullet}(P,M)$ 
corresponds to the endomorphism $\delta$ of the right hand side 
which is defined in the following way.

Let $g = (g_{i})_{i \in \ZZ} $ be a homogeneous element of the above graded module of 
cohomological degree $n$. 
Then the $i$-th component $\delta(g)_{i}$ of $\delta(g)$ defined to be 
\[
\delta(g)_{i} := \partial_{\Hom}(g_{i}) - (-1)^{n} g_{i-1} \otimes C \circ \sfq_{i}
\]
where $\partial_{\Hom}$ is the differential of $\Hom_{\Lambda}^{\bullet}(\frkp_{i}P, M_{i})$. 
 Therefore, we have the following lemma. 
 
\begin{lemma}\label{Hom complex lemma projective}
$\Phi$  gives an isomorphism between the complexes 
\[\Hom_{\GrMod A}^{\bullet}(P, M) \xrightarrow{\cong} 
\left(\prod_{i \in \ZZ} \Hom_{\Lambda}^{\#}(\frkp_{i}P, M_{i}), \delta \right).\]
\end{lemma}

\begin{proof}[Proof of Lemma \ref{projectively cofibrant lemma}]
(1) 
For $i  \in \ZZ$ and $M \in \sfC(\Mod \Lambda)$, 
we denote by $I_{i}(M) \in \sfC(\GrMod A)$ the complex such that $I_{i}(M)_{i} = M $ and $I_{i}(M)_{j} = 0$ 
for $j \neq i$.  
If $M$ is acyclic, then so is $I_{i}(M)$. 
We have an isomorphism 
$\Hom_{\Lambda}^{\bullet}(P, M) \cong \Hom^{\bullet}_{\GrMod A}(\frkp_{i} P ,I_{i} (M))$  of $\Hom$-complexes. 
Therefore if $P$ is homotopically projective, then so is $\frkp_{i} P$. 

(2) 
Let $M \in \sfC(\GrMod A)$ be an acyclic complex.
By description of $\delta$,  
via $\Phi^{-1}$,  the family 
 $\{ F_{j} := \prod_{i \geq j} \Hom_{\Lambda}^{\#}(\frkp_{i}P, M_{i}) \mid j \in \ZZ \}$ 
 gives a decreasing filtration of the complex $\Hom^{\bullet}_{\GrMod A}(P, M)$, 
 which is exhaustive by assumption. 
 Since $F_{j}/F_{j+1} = \Hom_{\Lambda}^{\bullet}(\frkp_{j}P, M_{j})$ is acyclic  for $j \in \ZZ$, 
 so is $\Hom^{\bullet}_{\GrMod A}(P,M)$. 
\end{proof}

Let $P' \in \sfC(\GrProj A)$ be another complex. 
We denote by $\sfq'$ what $\sfq$ for $P'$. 
Then, via the isomorphism induced from $\Phi$
\[
\Hom_{\GrMod A}^{\#}(P,P') \xrightarrow{\cong} 
\prod_{ i \in \ZZ} 
\Hom_{\Lambda}^{\#}(\frkp_{i}P, \frkp_{i-1}P' \otimes C) \oplus 
\prod_{ i \in \ZZ} 
\Hom_{\Lambda}^{\#}(\frkp_{i}P, \frkp_{i }P' ) 
\]
the canonical differential  of the left hand side 
corresponds to $\begin{pmatrix} \partial_{\Hom} & F \\ 0 & \partial_{\Hom}\end{pmatrix}$ 
where two $\partial_{\Hom}$'s are the canonical differential of $\Hom$-complexes and 
$F$ is a morphism  of degree $1$ defined in the following way.

Let $g= (g_{i})_{i \in \ZZ}$ be a homogeneous element of $\prod_{i \in \ZZ}\Hom_{\Lambda}^{n}(\frkp_{i} P, \frkp_{i}P')$ 
of degree $n$. 
Then the $i$-th component $F(g)_{i}$ of $F(g)$ is defined to be 
\[
F(g)_{i} := \sfq'\circ g_{i} - (-1)^{n}(g_{i -1} \otimes C) \circ \sfq. 
\]
 
Therefore, canonical inclusion and canonical projection yield an exact sequence in $\sfC(\Mod \kk)$ 
which splits if we forget the differentials
\begin{equation}\label{Hom complex exact sequence projective}
0 \to \prod_{i \in \ZZ}\Hom_{\Lambda}^{\bullet}(\frkp_{i} P, \frkp_{i -1} P' \otimes_{\Lambda}C) 
\to \Hom_{\GrMod A}^{\bullet}(P, P') 
\xrightarrow{\  (\frkp_{i})} \prod_{i \in \ZZ} \Hom_{\Lambda}^{\bullet}(\frkp_{i}P, \frkp_{i}P')
\to 0.
\end{equation}
Moreover, 
the complex $\Hom_{\GrMod A}(P,P')$ is the cone of the morphism 
\[F: \prod_{i \in \ZZ} \Hom_{\Lambda}^{\bullet}(\frkp_{i} P, \frkp_{i} P')[-1] 
\to \prod_{i \in \ZZ} \Hom_{\Lambda}^{\bullet}(\frkp_{i}P, \frkp_{i-1}P' \otimes C).\]

\begin{lemma}\label{frkp lemma} 
\begin{enumerate}[(1)] 
\item A complex $P \in \sfC(\GrProj A)$ is null-homotopic 
if and only if  so is $\frkp_{i} P $ 
for $ i \in \ZZ$. 

\item 
A morphism $f: P \to P'$ is homotopy equivalence 
if and only if
so is $\frkp_{i}(f)$ 
for $i \in\ZZ$. 
\end{enumerate}
\end{lemma} 

\begin{proof}
Since the functor $\frkp_{i} : \sfK(\GrProj A) \to \sfK(\Proj \Lambda)$ 
is  triangulated, ``only if" part of (1) and (2) are clear.  

We prove ``if" part of (1). 
Recall that a complex $X \in \sfC(\Mod \Lambda)$ is null-homotopic 
if and only if $\Hom_{\Lambda}^{\bullet}(X, Y)$ is acyclic for $Y \in \sfC(\Mod \Lambda)$ 
if and only if $\Hom_{\Lambda}^{\bullet}(X,X)$ is acyclic 
and the same is true in graded situation. 
Now, it is clear that ``if" part of (1) 
follows from the exact sequence (\ref{Hom complex exact sequence projective}). 

We prove ``if" part of (2).  
By assumption  $\frkp_{i}(\cone f) \cong \cone( \frkp_{i} f) $ is null-homotopic. 
Hence, the cone $\cone(f)$ is null-homotopic by (1). 
\end{proof}

\subsubsection{Construction of  graded projective complexes and morphisms between them}\label{a description of graded projective complexes} 


Let $ (\{{Q_{i} \}}_{i \in \ZZ} , \{\sfq_{i} \}_{i \in \ZZ} )$  
be a pair  of 
a collection of objects $Q_{i} \in \sfC(\Proj \Lambda)$ and 
a collection of morphisms $\sfq_{i} : Q_{i} \to Q_{i-1} \otimes_{\Lambda} C[1]$ in $\sfC(\Mod \Lambda)$.

We set  
\[
P^{n} := \bigoplus_{ i\in \ZZ } Q_{i}^{n} \otimes_{\Lambda}A(-i), \ \ \ 
d_{P}^{n}:=
\begin{pmatrix} 
\partial^{n}_{i-1} \otimes C & \sfq_{i}^{n} \\
0 & \partial^{n}_{i} \end{pmatrix}
 \]
where $\partial^{n}_{i}$ denotes the $n$-th differential of $Q_{i}$. 
 Then the pair  $P = (\bigoplus_{n \in \ZZ} P^{n}, \{d_{P}^{n}\})$  is a complex  of graded projective $A$-modules 
 such that $\frkp_{i}P = Q_{i}$ and $\frkq_{i}(d_{P}) = \sfq_{i}$.


Let $P'$ be the object of $\sfC(\GrProj A)$ associated to 
another pair $ (\{  Q'_{i} \}_{i \in \ZZ} , \{\sfq'_{i} \}_{i \in \ZZ} )$. 

\begin{lemma}\label{hom construction lemma}
Let $\{ g_{i}: Q_{i} \to Q'_{i}\}$ be a collection of morphisms in $\sfC(\Proj \Lambda)$. 
Then the following conditions are equivalent.
\begin{enumerate}[(1)]
\item 
 There exists a morphism $f: P \to P'$ in $\sfC(\GrProj A)$ such that $\frkp_{i}(f) =g_{i}$ for $i \in \ZZ$ 

\item The morphisms $\sfq'_{i} \circ g_{i}$ and $(g_{i-1} \otimes C[1]) \circ \sfq_{i}$ are homotopic to each other. 
\end{enumerate}
\end{lemma}

\begin{proof}
Taking the cohomology long exact sequence of 
 (\ref{Hom complex exact sequence projective}), 
we obtain the exact sequence of $\Hom$-spaces of cochain complexes
\[
\Hom_{\GrMod A}(P, P') 
\xrightarrow{\ \tuH((\frkp_{i}))} \prod_{i \in \ZZ} \Hom_{\Lambda}(Q_{i}, Q_{i}')
\xrightarrow{\tuH(F)} \prod_{i \in \ZZ}\Hom_{\Lambda}(Q_{i}, Q_{i-1}' \otimes_{\Lambda}C[1]).  
\]  
Therefore, 
the condition (1) is equivalent to the condition that 
there exists $h \in \prod_{i \in \ZZ}\Hom_{\Lambda}^{-1}(Q_{i}, Q_{i-1}\otimes C[1])$ 
such that $F(g) = \partial_{\Hom}(h)$ where $g := (g_{i})_{i \in \ZZ}$. 
It can be checked that 
the latter condition is equivalent to the condition (2).
\end{proof}

Combining Lemma \ref{frkp lemma} and Lemma \ref{hom construction lemma}, 
we obtain the following lemma and corollary.

\begin{lemma}\label{homotopic lemma}
Let $P \in \sfC(\GrProj A)$. 
Assume that 
a complex $Q'_{i} \in \sfC(\Proj \Lambda)$ 
homotopic to $\frkp_{i} P$ is given for all $i \in \ZZ$. 
Then, there exists  $P' \in \sfC(\GrProj A)$ homotopic to $P$ 
such that $\frkp_{i}P' = Q'_{i}$ in $\sfC(\Proj \Lambda)$ for $i \in \ZZ$.
\end{lemma}

\begin{proof}
Let $g_{i}: \frkp_{i}P \to Q_{i}'$ be a homotopy equivalence and 
$h_{i}: Q'_{i} \to \frkp_{i}P$ be a homotopy inverse. 
We define $\sfq'_{i}$ to be  the composite morphism below 
\[
\sfq'_{i}: Q'_{i} \xrightarrow{h_{i}} P_{i} \xrightarrow{\sfq_{i}} P_{i-1}\otimes_{\Lambda} C[1] 
\xrightarrow{ g_{i} \otimes C[1] } Q'_{i-1} \otimes_{\Lambda}C[1].
\] 
Then 
the morphisms $\sfq'_{i} \circ g_{i}$ and $(g_{i-1} \otimes C[1]) \circ \sfq_{i}$ are homotopic to each other. 
We denote by $P'$ 
the complex  constructed from the collection $( \{Q'_{i}\}_{i \in \ZZ}, \{ \sfq'_{i} \}_{i\in \ZZ})$  
By Lemma \ref{hom construction lemma}, there exists a morphism $f: P \to P'$ in $\sfC(\GrProj A)$ 
such that $\frkp_{i}(f) = g_{i}$. 
By Lemma \ref{frkp lemma}, $f$ is a homotopy equivalence. 
\end{proof}

\begin{corollary}\label{homotopic corollary}
An object $P \in \sfK(\GrProj A)$ belongs to $\sfK^{\mrb}(\grproj A)$ 
if and only if 
$\frkp_{i} P$ belongs to $\sfK^{\mrb}(\proj \Lambda)$ for all $i \in \ZZ$ 
and $\frkp_{i} P = 0 $ in $\sfK(\Proj \Lambda)$ for $|i| \gg 0$. 
\end{corollary}

\subsection{Projective dimension formula}\label{Projective dimension formula}

Let $M$ be a graded $A$-module. 
Then the graded $A$-module structure induces 
a $\Lambda$-module homomorphism $\tilde{\xi}_{M, i} : M_{i} \otimes_{\Lambda} C \to M_{i +1}$ 
for $i \in \ZZ$, which we call the \textit{action morphism} of $M$. 

If a graded $A$-module $M$ satisfies $M_{< 0} = 0$, 
then $(M \otimes_{A} C)_{0} \cong M_{0} \otimes_{\Lambda} C$. 
Moreover, applying $M \otimes_{A} -$ to the canonical inclusion $\mathsf{can} : C \hookrightarrow A(1)$ 
and looking at the degree $0$-part, 
we obtain the $0$-th action morphism 
\[
\tilde{\xi}_{M,0}: M_{0} \otimes_{\Lambda} C \cong (M \otimes_{A} C)_{0} 
\xrightarrow{  \   (M \otimes \mathsf{can})_{0} \ } 
(M \otimes_{A} A (1))_{0} \cong M_{1}.
\]

We remark that since these constructions are natural, it works for a complex $M \in \sfC(\GrMod A)$. 

Let $M$ be an object of $\sfD(\GrMod A)$.
We denote a representative of $M$ by the same symbol $M \in \sfC(\GrMod A)$. 
We define a morphism $\Xi_{M}^{0}: M_{0} \lotimes_{\Lambda} C \to M_{1}$ in $\sfD(\Mod \Lambda)$ 
to be the composite morphism 
\[
\Xi^{0}_{M}: M_{0} \lotimes_{\Lambda} C \xrightarrow{ \ \mathsf{can} \  } M_{0} \otimes_{\Lambda} C 
\xrightarrow{\ \tilde{\xi}_{M, 0} \ } M_{1}
\]
where $\mathsf{can}$ is a canonical morphism. 
It is easy to see that this morphism is well-defined. 

For $a \geq 1$ and $L \in \sfD(\Mod \Lambda)$, we set 
\[
L \lotimes_{\Lambda} C^{a} :=( \cdots ( ( L \lotimes_{\Lambda} C) \lotimes_{\Lambda}C)  \cdots \lotimes_{\Lambda} C ) 
\ \ \ \ ( a \textup{-times}).  
\] 
We also set $L \lotimes_{\Lambda} C^{0} := L$.

\begin{remark}\label{bimodule remark projective}
The above notation is only the abbreviation of the iterated application of the endofunctor $- \lotimes_{\Lambda}C$ 
of $\sfD(\Mod \Lambda)$. 
We will not deal with the complex $C^{a} =C \lotimes_{\Lambda} C \lotimes_{\Lambda} \cdots \lotimes_{\Lambda}C$ 
of $\Lambda$-$\Lambda$-bimodules. 
If $\Lambda$ is a projective module over a base ring $\kk$, 
then the standard definition for $C^{a}$ works well. 
For example, we have $(M \lotimes_{\Lambda} C) \lotimes_{\Lambda} C = M \lotimes_{\Lambda} (C \lotimes_{\Lambda} C)$.  
However, in general,  the standard definition does not work well.  
To settle such a  situation we need to use differential graded algebras. 
See \cite[Remark 1.12]{Yekutieli:Dualizing complexes} for similar discussion.  
\end{remark}

The following is the main theorem of this section.
For simplicity we set   $\Xi_{M}^{a}:= \Xi_{M}^{0} \lotimes_{\Lambda} C^{a}$. 
\[
\Xi_{M}^{a}: M_{0} \lotimes_{\Lambda} C^{a + 1}  \to M_{1} \lotimes_{\Lambda} C^{a}.
\]

\begin{theorem}\label{projective dimension formula}
Let $M \in \sfD(\GrMod  A)$ such that $M_{i}= 0$ for $i \neq 0,1$. 
Then we have 
\[
\pd\limits_{A} M = 
\sup \{ \pd_{\Lambda} M_{0}, \  \pd\limits_{\Lambda} (\cone \Xi^{a}_{M}) +a \mid a \geq 0 \}.
\]
In particular, 
we have $\pd\limits_{A} M < \infty$ 
if and only if the following conditions are satisfied: 
\begin{enumerate}[(1)]
\item $\pd_{\Lambda} M_{0} < \infty$. 
\item  $\pd\limits_{\Lambda} (\cone \Xi^{a}_{M}) < \infty$ for $a\geq 0$. 
\item $\Xi^{a}_{M}$ is an isomorphism for $a \gg 0$. 
\end{enumerate}
\end{theorem}

As a corollary, we obtain the following formula, from which the projective version of Theorem \ref{introduction theorem 1} follows. 

\begin{corollary}\label{projective dimension formula corollary}
For $M \in \sfD(\Mod  \Lambda)$, 
we have 
\[
\pd\limits_{A} M = \sup \{ \pd\limits_{\Lambda} ( M \lotimes_{\Lambda}C^{a} ) +a  \mid a \geq 0 \}.
\]
\end{corollary}

Combining this corollary with Proposition \ref{global dimension lemma}, we see the next corollary.  

\begin{corollary}\label{projective dimension formula corollary 2}
We have the equation 
\[
\gldim A  = \sup \{ \pd_{\Lambda}M \lotimes_{\Lambda} C^{a} + a \mid  M \in \Mod \Lambda, \ a \geq 0 \}. 
\]
\end{corollary}

We need preparations. 

\begin{lemma}\label{projective resolution lemma}\label{degree lemma}
Let $M$ be an object of  $\sfD(\Mod A)$ such that $M_{<0} = 0$ and 
$P\in \sfC(\GrProj A)$ a projective resolution of  $M$. 

Then the following assertions hold. 
\begin{enumerate}[(1)]

\item $\frkt_{< 0} P$ is null-homotopic and hence $\frkt_{\geq 0} P$ is homotopic to $P$ via the canonical morphism. 

\item 
The complex $\frkp_{0} P $ is a projective resolution of $M_{0}$ as complexes of $\Lambda$-modules.

\item 
For $N \in \sfD(\Mod \Lambda^{\op})$, 
 we have  
$(M \lotimes_{A} N )_{i} = 0$ for $i < 0$ and 
$(M \lotimes_{A} N )_{0} = M \lotimes_{\Lambda} N$. 

\item 
We have $\frkp_{i} P = (M \lotimes_{A} \Lambda)_{i}$ in $\sfD(\Mod \Lambda)$. 
\end{enumerate}
\end{lemma}

\begin{proof}We only prove (1). Proof of the other statements are left to reader. 
First note that for a complex $ L \in \sfC(\GrMod A)$ such that $L_{< 0} = 0$, 
we have $\Hom_{\sfD(\GrMod A)}(\frkt_{< 0} P, L ) = \Hom_{\sfK(\GrMod A)}(\frkt_{< 0} P, L ) = 0$ 
by Lemma \ref{Hom complex lemma projective}. 
Since $P$ is quasi-isomorphic to $M$, 
we have 
\[
\Hom_{\sfK(\GrMod A)}(\frkt_{< 0} P, P ) = \Hom_{\sfD(\GrMod A)}(\frkt_{< 0} P, P) = \Hom_{\sfD(\GrMod A)}(\frkt_{< 0} P, M) = 0.
\] 
Therefore, using the exact triangle (\ref{ad hoc ref 3/20})
we see that $\frkt_{< 0} P$ is a direct summand of $\frkt_{ \geq 0} P[-1]$ in $\sfK(\GrMod A)$ .
 Since $\Hom_{\sfK(\GrMod A)}(\frkt_{< 0} P, \frkt_{\geq 0} P[-1]) =0$, 
 we conclude that $\frkt_{ < 0} P = 0$  in $\sfK(\GrMod A)$. \end{proof}

The following lemma gives a relationship between $\frkp_{i} P$ and $\Xi_{M}^{a}$.

\begin{lemma}\label{principal lemma}
Let $M$ be an object of  $\sfD(\GrMod A)$ such that $M_{i} = 0$ for $i \neq 0,1$ and  
$P\in \sfC(\GrProj A)$ a projective resolution of  $M$. 
Then in $\sfD(\Mod \Lambda)$ we have 
\[
\frkp_{i}P \cong 
\begin{cases} 
\cone \Xi_{M}^{i-1}[i-1] & i  \geq 1,\\ 
M_{0}  & i= 0, \\
0 & i < 0.
\end{cases} 
\]
\end{lemma}

\begin{proof}
The case $i \leq 0$ is clear from Lemma \ref{projective resolution lemma}. 

Applying $M \lotimes_{A} - \lotimes_{\Lambda}C^{a}$ to the canonical exact sequence 
$ 0 \to C(-1) \to A \to \Lambda \to 0$, we obtain an exact triangle 
\[
M\lotimes_{A} C^{a+1}(-1) \to M \lotimes_{\Lambda}C^{a } \to  M \lotimes_{A} C^{a} \to. 
\]
We remark that the degree $1$-part of  the first morphism 
is identified with 
\[
\Xi_{M}^{a}  : M_{0} \lotimes_{\Lambda} C^{a + 1} \cong (M \lotimes_{A} C^{a +1}( -1) )_{1} \to 
( M \lotimes_{\Lambda} C^{a})_{1}    \cong M_{1} \lotimes_{\Lambda} C^{a}. 
\]
By looking degree $i\geq 1$ part, we obtain isomorphisms
\[
\begin{split}
(M \lotimes_{A} C^{a})_{1}  &\cong \cone[\Xi_{M}^{a} :  M_{0}  \lotimes_{\Lambda}C^{a+1} \to M_{1} \lotimes_{\Lambda} C^{a} ],  \\
(M \lotimes_{A} C^{a})_{i}  &\cong  (M \lotimes_{A} C^{a+1})_{ i -1}[1] \textup{ for } i \geq 2.
\end{split}\]
Therefore, for $i \geq 1$ we have 
\[
\begin{split} 
\frkp_{i} P &\cong  (M \lotimes_{A} \Lambda)_{i}  \cong (M \lotimes_{A} C)_{ i-1}[1] \cong  (M \lotimes_{A} C^{2})_{ i-2}[2] \\
&\cong \cdots  \cong (M \lotimes_{A} C^{i-1})_{ 1}[ i-1] =   \cone \Xi_{M}^{i-1}[i-1].
\end{split}
\]
\end{proof}

\begin{proof}[Proof of Theorem \ref{projective dimension formula}] 
For simplicity we set 
\[
m := \pd\limits_{A} M, \ \ 
n := \sup \{ \pd_{\Lambda} M_{0}, \  \pd\limits_{\Lambda} (\cone \Xi^{a}_{M}) +a \mid a \geq 0 \}.
\]

We prove $m \geq n$. We may assume that $m < \infty$. 
Let $P$ be a projective resolution of $M$ as a complex of graded $A$-modules 
such that $P^{< m} = 0$. 
It follows from  Lemma \ref{projectively cofibrant lemma} and Lemma \ref{principal lemma}, 
that $\frkp_{0} P $ is a projective resolution of $M_{0}$ and 
and that $\frkp_{i} P $ is a projective resolution of $\cone \Xi_{M}^{i-1}[i-1]$ for $i \geq 1$. 
Hence we prove the desired inequality. 

We prove $ m \leq n$. 
We may assume that $n < \infty$. 
Let $Q_{0}$ be a projective resolution of $M_{0}$ in $\sfC(\Mod \Lambda)$ such that $Q_{0}^{< n} =0$ 
and $Q_{i}$ a projective resolution of $\cone \Xi_{M}^{i-1}[i-1]$ in $\sfC(\Mod \Lambda)$ 
such that $Q_{i}^{< n} = 0$ for $i \geq 1$. 
It follows from Lemma \ref{projectively cofibrant lemma}, Lemma \ref{homotopic lemma} and Lemma \ref{principal lemma} 
that there exists a projective resolution $P \in \sfC(\GrMod A)$ of $M$ 
such that $\frkp_{i} P = Q_{i}$ for $i \geq 0$ and $\frkp_{i} P = 0$ for $i < 0$. 
Since $P^{< n} = 0$, we prove the desired inequality.
\end{proof}

\subsubsection{Finiteness of global dimension}

Using Corollary \ref{projective dimension formula corollary} we can deduce the following result. 

\begin{corollary}\label{criterion for finite global dimension}
Let $A = \Lambda \oplus C$ be a trivial extension algebra. 
Then the following conditions are equivalent. 
\begin{enumerate}[(1)]
\item $\gldim A < \infty$. 
\item $\gldim \Lambda < \infty$ and $C^{a} = 0$ for some  $a > 0$. 
\end{enumerate}
\end{corollary}

\begin{proof}
(1) $\Rightarrow$ (2). 
Since $\gldim \Lambda \leq \gldim A$, we have $\gldim \Lambda < \infty$. 
Applying Corollary \ref{projective dimension formula corollary} to $\Lambda_{A}$, 
we see that $C^{a} = 0$ for $a > \gldim A$.

(2) $\Rightarrow$ (1). 
Let $M$ be a $\Lambda$-module. 
For $b \in \NN$ 
we have the following inequality by Theorem \ref{Avramov-Foxby:inequality},  
\[
\pd_{\Lambda}M \lotimes C^{b}  \leq \pd_{\Lambda} M + b \pd_{\Lambda} C.  
\]
Therefore, if $C^{a} = 0$, then we have 
\[
\pd_{A} M = \sup\{ \pd_{\Lambda}M \lotimes C^{b} +b  \mid 0 \leq b \leq a -1 \} \leq  a(\gldim \Lambda + 1) -1. 
\]
 Hence by Proposition \ref{global dimension lemma},  $\gldim A \leq  a(\gldim \Lambda + 1) -1 $. 
\end{proof}

We can reprove the characterization of the case $\gldim A \leq 1$ which was first proved by I. Reiten 
\cite[Proposition 2.3.3]{Reiten:Thesis} (see also \cite[Corollary 5.3]{Lofwall}). 
The details of proof is left to the readers.

\begin{corollary}
We have $\gldim A \leq 1$ 
if and only if the following conditions are satisfied: 
\begin{enumerate}[(1)]
\item $\gldim \Lambda \leq 1$. 
\item The left $\Lambda$-module ${}_{\Lambda}C$ is flat. 
\item $M \otimes_{\Lambda} C$ is projective for $M \in \Mod \Lambda$. 
\item $C \otimes_{\Lambda} C = 0$. 
\end{enumerate}
\end{corollary}

\subsection{A criterion of perfectness of  complexes  of graded $A$-modules}



 Recall that  an object of $\sfK^{\mrb}(\proj \Lambda)$ (resp.  $\sfK^{\mrb}(\grproj A)$) 
is called  a perfect complex (resp. graded perfect complex).
Using the same method with the proof of Theorem \ref{projective dimension formula}, 
we can prove a criterion that an object  $ M \in \sfD(\Mod \Lambda)$ is a graded perfect complex.

\begin{theorem}\label{a criterion of perfectness}  
An object $M \in \sfD(\Mod \Lambda)$ belongs to $\sfK^{\mrb}(\grproj A)$ 
if and only if $M \lotimes_{\Lambda} C^{a}$ belongs to 
$\sfK^{\mrb}(\proj \Lambda)$ for $a \geq 0$ 
and $M \lotimes_{\Lambda} C^{a} = 0$ for $a \gg 0$. 
\end{theorem}

\begin{proof}
In the case where $M_{i} = 0$ for $i \neq 0$, 
Lemma \ref{principal lemma} become 
\[
\frkp_{i}P \cong 
\begin{cases} 
M \lotimes_{\Lambda} C^{i}[i] & i \geq 0,\\ 
0 & i < 0.
\end{cases} 
\]
Now the proof is the same with that of Theorem \ref{projective dimension formula} 
except that  we need to use  Corollary \ref{homotopic corollary} instead of  Lemma \ref{homotopic lemma}.
\end{proof}

\subsubsection{Applications}

Assume that $C_{\Lambda}$ belongs to $\sfK^{\mrb}(\proj \Lambda)$.  
Then the endofunctor $- \lotimes_{\Lambda} C^{a}$ of $\sfD(\Mod \Lambda)$ 
preserves $\sfK^{\mrb}(\proj \Lambda)$ for $a \geq 0$. 
We denote the restriction functor by $( -\lotimes_{\Lambda} C^{a})|_{\sfK}$.  

Observe that there exists the following increasing sequence of  thick subcategories of $\sfK^{\mrb}(\proj \Lambda)$. 
\[
\Ker(- \lotimes_{\Lambda}C)|_{\sfK} \subset 
\Ker(- \lotimes_{\Lambda}C^{2})|_{\sfK} \subset \cdots 
\subset 
\Ker(- \lotimes_{\Lambda}C^{a})|_{\sfK} \subset \cdots.  
\]

Applying Theorem \ref{a criterion of perfectness}, 
we obtain the following result.

\begin{corollary}\label{description of kernel} 
Under the above situation, we have 
\[
\sfD(\Mod \Lambda) \cap \sfK^{\mrb}(\grproj A) 
= \sfK^{\mrb}(\proj \Lambda) \cap \sfK^{\mrb}(\grproj A)
 = \bigcup_{a \geq 0} \Ker  (- \lotimes_{\Lambda}C^{a})|_{\sfK}.
\]
\end{corollary}

This result plays an important role in \cite{higehaji}, 
since in the case where $A$ is Noetherian, 
it gives a description of the kernel of the canonical functor 
\[
\varpi: \sfD^{\mrb}(\mod \Lambda) \to \sfD^{\mrb}(\grmod A) \to \grSing A = \sfD^{\mrb}(\grmod A) / \sfK^{\mrb}(\grproj A)
\]
where $\grSing A$ is the graded singular derived category of $A$.

As in the same way of Corollary \ref{criterion for finite global dimension}, 
we can prove the following Corollary by using Theorem \ref{a criterion of perfectness}.

\begin{corollary}
Assume that $A =\Lambda \oplus C$ is Noetherian. 
Then the following conditions are equivalent.  
\begin{enumerate}[(1)]
\item 
$\sfK^{\mrb}(\grproj A) = \sfD^{\mrb}(\grmod A)$.

\item 
$\sfK^{\mrb}(\proj \Lambda) = \sfD^{\mrb}(\mod \Lambda)$ 
and $C^{a} =  0$ for some $a > 0$. 
\end{enumerate}
\end{corollary}

\begin{proof}
Assume that the condition (1) holds. 
Then since $\sfD^{\mrb}(\mod \Lambda) \subset \sfD^{\mrb}(\grmod A)$, 
the first condition of (2) follows from Theorem \ref{a criterion of perfectness}. 
Applying Theorem \ref{a criterion of perfectness} to $M = \Lambda$, 
we obtain $C^{a} = 0$ for $a \gg 0$.

We prove (2) $\Rightarrow$ (1). 
Since $C_{\Lambda}$ is finitely generated, 
the derived tensor product  $ ( - \lotimes_{\Lambda}C)|_{\sfK}$ 
has its value in $\sfD^{\mrb}(\mod \Lambda)$. 
It follows from the first condition  that   
the functor $- \lotimes_{\Lambda}C$ is an endofunctor of  $\sfD^{\mrb}(\mod \Lambda)$. 
Therefore,    
$\sfD^{\mrb}(\mod \Lambda)$ is contained in $\sfK^{\mrb}(\grproj A)$ by Theorem \ref{a criterion of perfectness}. 
Since the graded degree shift $(1)$ is an exact autofunctor of $\sfK^{\mrb}(\grproj A)$,  
$\sfD^{\mrb}(\mod \Lambda)(i)$ is also  contained in $\sfK^{\mrb}(\grproj A)$ for $ i \in \ZZ$.  
Since every object of $\sfD^{\mrb}(\grmod A)$ is constructed by extensions from objects of 
$\{ \sfD^{\mrb}(\mod \Lambda)(i) \mid i \in \ZZ \}$, 
it is contained in $\sfK^{\mrb}(\grproj A)$. 
\end{proof}

\begin{remark}\label{general theorem projective}
It seems likely that we can generalize  Theorem \ref{projective dimension formula} 
for an object $M \in \sfD(\GrMod A)$ such that $M_{i} = 0$ for $|i| \gg 0$ 
by using the derived action map. 

The graded $A$-module structure on $M$ induces 
the derived action morphism $\rho_{i}: M_{i} \lotimes_{\Lambda} C \to M_{i+1}$ 
for $i \in \ZZ$. 
Note that we have $\rho_{i+1} \circ (\rho_{i} \lotimes_{\Lambda} C) = 0$. 
Hence we obtain a complex $\cX_{a}$ of objects of $\sfD^{\mrb}(\Mod \Lambda)$ 
\[
\cX_{a}: 
M_{i_{0}}\lotimes_{\Lambda} C^{a} \xrightarrow{ \rho_{i_{0}} \lotimes C^{a-1} } M_{i_{0} +1} \lotimes C^{a-1} 
\xrightarrow{ \rho_{i_{0}-1} \lotimes C^{a-2}} M_{i_{0}+2} \lotimes_{\Lambda} C^{a-2} \to \cdots \xrightarrow{\rho_{a-1}} M_{a} 
\]
where we set  $i_{0}:= \min\{ i \mid M_{i} \neq 0\}$. 
Using  the totalization $\tot(\cX_{a})$ of the complex $\cX_{a}$, 
we may be able to  establish a projective dimension formula for $M$. 
However we should care about a subtlety of the totalization of a complex of objects of a triangulated category (see \cite{GM}).

We are content with Theorem \ref{a criterion of perfectness} which is sufficient for our applications. 
\end{remark}

\section{Injective dimension formula}\label{A criterion for finiteness of injective dimension}

In this Section \ref{A criterion for finiteness of injective dimension}, 
again, $\Lambda$ denotes an algebra, $C$ denotes a bimodule over it 
and $A = \Lambda \oplus C$ is the trivial extension algebra. 
We establish the injective  dimension formula for  
an object $M \in \sfD^{\mrb}(\GrMod A)$ such that $M_{i} = 0$ for $i \neq 0,1$. 
The key tool is a decomposition of a complex $I \in \sfC(\GrInj A)$ 
of graded injective $A$-modules 
according to degree of the cogenerators.  

Since the arguments are dual of that of Section \ref{A criterion of perfectness},  
we leave the details of proofs to the readers. 

In Section \ref{A criterion of Iwanaga-Gorensteinness}, we  discuss  a condition that $A$ is an Iwanaga-Gorenstein algebra.

\subsection{Decomposition of complexes of graded injective $A$-modules }\label{Decomposition of graded injective $A$-modules}

\subsubsection{Decomposition of graded injective $A$-modules}

Let $I$ be a graded injective $A$-module.
Recall that 
$\frki_{i}I := \grHom_{A}(\Lambda, I)_{i}, \ \ \frks_{i}I := \grHom_{\Lambda}(A,  \frki_{i}I)(-i)$.  
Since now $A= \Lambda \oplus C$  is a trivial extension algebra, 
we have an isomorphism $\frks_{i} I \cong \Hom_{\Lambda}(C, \frki_{i} I) ( -i +1) \oplus (\frki_{i} I )(-i)$ 
as graded $\Lambda$-modules.

By Lemma \ref{decomposition of graded injective modules}, 
we have the following isomorphisms
\begin{equation}\label{decomposition of graded injective modules 2}
\begin{split}
 I  &\cong \bigoplus_{ i\in \ZZ} \frks_{i} I \ (\textup{as graded $A$-modules}), \\
 I_{i} &\cong (\frks_{i} I)_{i}  \oplus  (\frks_{i+1} I)_{i} \cong  
\frki_{i} I  \oplus  \Hom_{\Lambda}(C, \frki_{i+1} I) \  \ \ (\textup{as $\Lambda$-modules}).
\end{split} 
\end{equation}

As in  Section \ref{Decomposition of graded projective $A$-module}, 
we may decompose the degree $i$-part $f_{i}$ of a morphism $f: I \to I'$ of graded injective modules under the isomorphism 
of (\ref{decomposition of graded injective modules 2}) 
\[
\begin{pmatrix} 
\frki_{i} (f)  & \frkj_{i}(f) \\
0 & \Hom_{\Lambda} ( C, \frki_{i +1}(f))
\end{pmatrix} 
: \frki_{i} I \oplus \Hom_{\Lambda} (C, \frki_{i +1}I ) 
\to 
 \frki_{i} I' \oplus \Hom_{\Lambda} (C, \frki_{i +1}I' ).  
\]

\subsubsection{Decomposition of complexes of graded injective $A$-modules }
As in Section \ref{A decomposition of  complexes of graded projective modules} 
and Section \ref{a description of graded projective complexes} 
we may construct from  a complex $I$ of graded injective $A$-modules 
a complex $\frki_{i} I$ of injective $\Lambda$-modules and  prove the following lemma. 

\begin{lemma}\label{injective lemma}
Let $I$ be a complex of graded injective $A$-modules. 
Then, 
the following assertions hold. 

\begin{enumerate}[(1)]
\item
We have an exact triangle  in $\sfK(\Mod \Lambda)$ (and in $\sfD(\Mod \Lambda)$) 
\[
\Hom_{\Lambda}(C, \frki_{i +1} I)[-1] \xrightarrow{\sfj_{i}} \frki_{i} I \to I_{i}  \to \Hom_{\Lambda}(C, \frki_{i +1} I) 
\]
where $\sfj_{i}$ is the morphism given by the collection $\{ \frkj_{i}(d_{I}^{n})\}_{n \in \ZZ}$. 
 
In particular, 
$\tuH(I)_{i} = 0$ if and only if the morphism $\sfj_{i}: \Hom_{\Lambda}(C, \frki_{i +1} I)[-1] \to \frki_{i} I $ 
is a quasi-isomorphism. 

\item 
Assume that a complex $J'_{i} \in \sfC(\Inj \Lambda)$ homotopic to $\frki_{i} I$ is given for $i \in \ZZ$. 
Then there exists  $I' \in \sfC(\GrInj A)$ homotopic to $I$ 
such that $\frki_{i} I' = J'_{i}$ in $\sfC(\Inj \Lambda)$ for $i \in \ZZ$. 
\end{enumerate}
\end{lemma}


\subsection{Injective dimension formula}\label{Injective dimension formula}

Let $M$ be a graded $A$-module. 
Then the graded $A$-module structure induces 
a $\Lambda$-module homomorphism $\tilde{\theta}_{M, i  } : M_{i}\to  \Hom_{\Lambda}(C , M_{i +1})$ 
for $i \in \ZZ$, which we call the \textit{coaction morphism} of $M$. 

If a graded $A$-module $M$ satisfies $M_{> 1} = 0$, 
then $ (\grHom_{A}(C, M))_{1} \cong \Hom_{\Lambda}(C, M_{1})$. 
Moreover, applying $\grHom_{A}(- ,M)$ to the canonical inclusion $\mathsf{can} : C \hookrightarrow A(1)$ 
and looking the degree $1$-part, 
we obtain the $1$-st coaction morphism 
\[
\tilde{\theta}_{M,0}: 
M_{0}  \cong  \grHom_{A}(A(1), M)_{1} 
\xrightarrow{  \   \grHom(\mathsf{can}, M)_{1} \ } 
\grHom_{A}(C, M )_{1} \cong \Hom_{\Lambda}(C, M_{1}). 
\]

We remark that since these constructions are natural, it works for a complex $M \in \sfC(\GrMod A)$. 

Let $M$ be an object of $\sfD(\GrMod A)$.
We denote a representative of $M$ by the same symbol $M \in \sfC(\GrMod A)$. 
We define a morphism $\Theta_{M}^{0}: M_{0} \to \RHom_{\Lambda}(C, M_{1})$ in $\sfD(\Mod \Lambda)$ 
to be the composite morphism 
\[
\Theta^{0}_{M}: 
M_{0} \xrightarrow{\ \tilde{\theta}_{M, 0}\  }  \Hom_{\Lambda}^{\bullet}(C, M_{1}) \xrightarrow{ \ \mathsf{can} \ } 
\RHom_{\Lambda}(C, M_{1}) 
\]
where $\mathsf{can}$ is a canonical morphism. 
It is easy to see that this morphism is well-defined. 

For $a \geq 1$ and $L \in \sfD(\Mod \Lambda)$, we set 
\[
\RHom_{\Lambda} (C^{a}, L) 
:= \RHom_{\Lambda}(C,  \cdots \RHom_{\Lambda}(C, \RHom_{\Lambda}(C, L)) \cdots ) 
\ \ \ \ ( a \textup{-times}).  
\] 
We also set $\RHom_{\Lambda}(C^{0}, L)  := L$.

\begin{remark}\label{bimodule remark injective}
The above notation is only the abbreviation of the iterated application of the functor $\RHom_{\Lambda}(C, -)$. 
We will not deal with the complex $C \lotimes_{\Lambda} C \lotimes_{\Lambda} \cdots \lotimes_{\Lambda}C$ 
of $\Lambda$-$\Lambda$-bimodules. 
\end{remark}

The following is the main theorem of this section.
For simplicity we set   $\Theta_{M}^{a}:= \RHom_{\Lambda}(C^{a}, \Theta_{M}^{0})$. 
\[
\Theta_{M}^{a} : \RHom_{\Lambda}(C^{a}, M_{0}) \to \RHom_{\Lambda}(C^{a+1}, M_{1})
\]

\begin{theorem}\label{injective dimension formula}
Let $M$ be a graded $A$-module such that $M_{i} = 0$ for $i \neq 0,1$. 
Then, 
\[
\grinjdim_{A} M = \sup\{\injdim_{\Lambda}  M_{1}, \ \injdim_{\Lambda} ( \cone \Theta_{M}^{a}) + a+ 1 \mid a \geq 0\}.
\]
In particular, 
 we have $\grinjdim_{A} M < \infty$ 
if and only if the following conditions are satisfied: 
\begin{enumerate}[(1)]
\item  $\injdim_{\Lambda} M_{1} < \infty $. 

\item
$\injdim \cone(\Theta_{M}^{a}) < \infty $ for  $a \geq 0$.

\item $\Theta_{M}^{a}$ is an isomorphism for $a \gg 0$.
\end{enumerate}
\end{theorem}

As a corollary, we obtain the following formula, from which Theorem \ref{introduction theorem 1} follows. 

\begin{corollary}\label{injective dimension formula corollary}
Let $M \in \sfD^{+}( \Mod \Lambda)$. 
Then, 
\[
\injdim_{A} M = \sup \{ \injdim_{\Lambda}  \RHom_{\Lambda}(C^{a}, M)  + a \mid a \geq 0 \}.
\]
\end{corollary}

Combining this corollary with Proposition \ref{global dimension lemma}, we see the next corollary.  

\begin{corollary}\label{injective dimension formula corollary 2}
We have the equation 
\[
\gldim A  =  \sup \{ \injdim_{\Lambda}(\RHom_{\Lambda}(C^{a}, M))  + a \mid M \in \Mod \Lambda, \ a \geq 0 \}.
\]
\end{corollary}

We can prove 
Theorem \ref{injective dimension formula} 
in the same way of Theorem \ref{projective dimension formula} 
by using following Lemmas, which are dual of Lemma \ref{degree lemma} and Lemma \ref{principal lemma}.

\begin{lemma}\label{degree lemma injective}
Let $M$ be 
 a graded $A$-module such that $M_{i} = 0$ for $i \neq 0,1$ 
 and $I$ an injective resolution of $M$ in $\GrMod A$. 
Then, 
\begin{enumerate}[(1)]

\item $\frks_{> 1} I$ is null-homotopic and hence $\frks_{ \leq 1} I$ is homotopic to $I$ via the canonical morphism. 

\item  $\frki_{1} I $ is an injective resolution  of $M_{1}$ as  complexes of $\Lambda$-modules.

\item We have $\grRHom_{A}(\Lambda, M)_{i} = \frki_{i} I$ in $\sfD(\Mod \Lambda)$.

\item 
For $L \in \sfD(\Mod \Lambda)$, we have $\grRHom_{A}(L, M)_{1} = \RHom_{\Lambda}(L, M_{1})$. 
\end{enumerate}
\end{lemma} 

\begin{lemma}\label{A key observation for the criterion}
Let $M$ be as in Theorem \ref{injective dimension formula} 
and $I \in \sfC(\GrInj A)$  an injective resolution of $M$.  
Then, we have the following isomorphisms in $\sfD(\Mod \Lambda)$. 
\[
\begin{split} 
&\frki_{1} I \cong M_{1}, \ \ 
\frki_{-i} I   \cong
\cone \left( \Theta_{M}^{i} \right)[-i-1] 
\textup{ for } i \geq 0
\end{split}\]
\end{lemma}

\begin{remark}\label{general theorem injective}
We give a remark similar to  Remark \ref{general theorem projective}. 
It seems likely that we can generalize  Theorem \ref{injective dimension formula} 
for an object $M \in \sfD(\GrMod A)$ such that $M_{i} = 0$ for $|i| \gg 0$ 
by using the derived coaction map.

By adjunction,  
the derived action morphism $\rho_{i}: M_{i} \lotimes_{\Lambda} C \to M_{i+1}$ 
induces 
the derived coaction morphism $\Theta'_{i}: M_{i} \to \RHom_{\Lambda}(C, M_{i +1})$.   
Note that we have $\RHom_{\Lambda}(C, \Theta'_{i + 1}) \circ \Theta'_{i}  = 0$. 
Hence we obtain a complex $\cY_{a}$ of objects of $\sfD^{\mrb}(\Mod \Lambda)$ 
\[
\cY_{a}: 
M_{j_{0}  -a} \xrightarrow{ \Theta'_{ j_{0} -a} } \RHom_{\Lambda} (C, M_{j_{0} -a +1}) 
\to \cdots \xrightarrow{\RHom(C^{a-1}, \Theta'_{j_{0} -1})} \RHom_{\Lambda}(C^{a}, M_{j_{0}})
\]
where we set  $j_{0}:= \max\{ i \mid M_{i} \neq 0\}$. 
Using  the totalization $\tot(\cY_{a})$ of the complex $\cY_{a}$, 
we may be able to  establish an injective dimension formula for $M$. 
However we should care about a subtlety of the totalization of a complex of objects of a triangulated category (see \cite{GM}).

We are content with Theorem \ref{injective dimension formula} which is sufficient for our applications. 
\end{remark}

\begin{remark}\label{projective and injective}
We have obtained formulas for projective dimension and injective dimension. 
If we only deal with a finite dimensional algebra $\Lambda$  over a field $\kk$  
and a suitably bounded complex of   finite dimensional modules over it,  
then we can deduce one from the other by using the natural isomorphism 
\[
M \lotimes_{\Lambda} \tuD(N ) \to \tuD\RHom_{\Lambda}(M,N) 
\] 
for $M,N \in \sfD^{-}(\mod \Lambda)$ 
where $\tuD(-) = \Hom_{\kk}(-, \kk)$ denotes the $\kk$-duality. 
\end{remark}

\subsection{A criterion of Iwanaga-Gorensteinness}\label{A criterion of Iwanaga-Gorensteinness}

In this Section \ref{A criterion of Iwanaga-Gorensteinness}, 
we discuss a condition that a trivial extension algebra $A = \Lambda \oplus C$ 
is an Iwanaga-Gorenstein (IG) algebra.

\subsubsection{Right asid bimodules}

Before recalling the definition of IG-algebra, 
we give  a condition that $A$ has finite right self-injective dimension. 
For this purpose, we give a description of the morphism $\Theta_{A}^{a}$.

We define  a morphism  $\tilde{\lambda}_{r}: \Lambda \to \Hom_{\Lambda}(C,C)$ by the formula 
$\tilde{\lambda}_{r}(x)(c) := xc$ for $x \in \Lambda$ and $c \in C$. 
We denote the composite morphism  $\lambda_{r} = \mathsf{can} \circ \tilde{\lambda}_{r}$ in $\sfD(\Mod \Lambda)$ 
where $\mathsf{can}$ is the canonical morphism $\Hom_{\Lambda}(C,C) \to \RHom_{\Lambda}(C,C)$
\[
\lambda_{r}: 
\Lambda \xrightarrow{\tilde{\lambda}_{r}} \Hom_{\Lambda}(C,C) \xrightarrow{\mathsf{can}} \RHom_{\Lambda}(C,C). 
\]
We can check that $\lambda_{r}=\Theta_{A}^{0}$. 
Therefore, by  Corollary \ref{finitely graded IG-lemma} and Theorem  \ref{injective dimension formula},  
we deduce the following result.

\begin{theorem}\label{right asid theorem} 
The following conditions are equivalent: 
\begin{enumerate}[(1)]
\item $\injdim A_{A} < \infty$.  

\item the following conditions are satisfied: 

\begin{enumerate}[ {Right} $\asid$ 1.]
\item $\injdim_{\Lambda} C < \infty$. 

\item $\injdim_{\Lambda} \cone (\RHom_{\Lambda}(C^{a}, \lambda_{r} )) < \infty $ 
for $a \geq 0$. 

\item  The morphism $\RHom_{\Lambda}(C^{a}, \lambda_{r} )$ is an isomorphism for $a \gg 0$. 
\end{enumerate}
\end{enumerate}
\end{theorem}

\begin{definition}\label{asid number definition}
\begin{enumerate}[(1)]
\item 
A $\Lambda$-$\Lambda$-bimodule $C$ is called a \textit{right asid (attaching self injective dimension) bimodule} if $\injdim A_{A} < \infty$. 

\item 
For a right asid bimodule $C$, we define the \textit{right asid number} $\alpha_{r}$ to be 
\[
\alpha_{r} := \min\{ a \geq 0 \mid \RHom_{\Lambda}(C^{a}, \lambda_{r} ) \textup{ is an isomorphism} \}. 
\]
\end{enumerate}
\end{definition}

By Lemma \ref{degree lemma injective} and Lemma \ref{A key observation for the criterion}, 
we obtain description of the right  asid numbers $\alpha_{r}$. 
 
\begin{corollary}\label{asid number corollary}
Let $\Lambda$ be an algebra, 
$C$ a right asid bimodule over $\Lambda$ 
and $A = \Lambda \oplus C$ the trivial extension algebra.  
We regard  a minimal injective resolution $I$ of $A$ as a complex. 
Then, we have 
\[
\alpha_{r}  = \max\{ a\geq -1 \mid \RHom_{A}(\Lambda, A)_{-a} \neq 0\} + 1 = \max \{ a \geq -1 \mid \frki_{-a} I \neq 0\} +1.
\]
\end{corollary}

Dually, 
we define the notion of  \textit{left asid bimodule}  and the \textit{left asid number} $\alpha_{\ell}$ for them 
by using the left version $\lambda_{\ell}$  of $\lambda_{r}$. 
\[
\lambda_{\ell}: \Lambda \to \RHom_{\Lambda^{\op}}(C, C). 
\]
A $\Lambda$-$\Lambda$-bimodule $C$ is called \textit{asid} if it is both right and left asid bimodules.



\subsubsection{A criterion of Iwanaga-Gorensteinness}

Recall that an algebra  is called IG if it is Noetherian on both sides and 
has finite self-injective dimension on both sides. 
We note that by Zaks' Theorem, 
right self-injective dimension and left self-injective dimension of 
an IG-algebra $A$ coincide. 

The notion of graded IG algebra is defined for graded algebras in the same way. 
It is easy to see that for a finitely graded algebra  $A= \bigoplus_{i = 0}^{\ell} A_{i}$ 
graded IG-ness and IG-ness as an ungraded algebra are the same condition.   
Moreover in this case, we have the following equations by Corollary \ref{finitely graded IG-lemma}
\[
\injdim_{A} A_{A} = \injdim_{A^{\op}} {}_{A} A = \grinjdim_{A} A_{A} = \grinjdim_{A^{\op}} {}_{A}A.
\]

It is worth mentioning the following fact  which is obtained as a combination of known results. 

\begin{proposition}
A (not necessary finitely) graded algebra $A = \bigoplus_{i \geq 0} A_{i}$ is graded IG 
if and only if it is (ungraded) IG. 
\end{proposition}

\begin{proof}
This is a consequence of the following two results. 
A graded ring $A$ is graded left Noetherian if and only if it is (ungraded) left Noetherian 
by \cite[II 3.1]{NV:Graded and Filtered}.
For a (graded) Noetherian ring $A$, 
we have $\injdim A_{A} \leq \grinjdim A_{A} \leq \injdim A_{A} + 1$. 
The first inequality is well-known. 
The second is due to Van den Bergh (see \cite[Theorem 12]{Yekutieli vdB theorem}).
\end{proof}

A  trivial extension algebra $A = \Lambda \oplus C$ is Noetherian on both sides 
if and only if  so is $\Lambda$ and 
$C_{\Lambda}$ and ${}_{\Lambda} C$ are finitely generated. 
Hence by Theorem \ref{right asid theorem} and its left version, 
we obtain the following criterion of IG-ness of  a trivial extension algebra.

\begin{theorem}\label{asid theorem 1}
Let $\Lambda$ be a Noetherian algebra and $C$ a $\Lambda$-$\Lambda$-bimodule 
which is finitely generated as right and  as left $\Lambda$-modules respectively. 
The trivial extension algebra $A = \Lambda \oplus C$ is an Iwanaga-Gorenstein algebra if and only if 
the following conditions are satisfied: 
\begin{enumerate}[(1)]

\item $C$ satisfies the conditions right $\asid$ 1,2,3. 

\item $C$ satisfies the left \asid   \ conditions below

\begin{enumerate}[Left $\asid$ 1.]
\item $\injdim_{\Lambda^{\op}} {}_{\Lambda}C < \infty$.

\item $\injdim_{\Lambda^{\op}} \cone \RHom_{\Lambda^{\op}}(C^{a}, \lambda_{\ell}) < \infty $ 
for $a \geq 0$. 

\item  The morphism $\RHom_{\Lambda^{\op}}(C^{a},\lambda_{\ell})$ is an isomorphism for $a \gg 0$. 
\end{enumerate}
\end{enumerate}
\end{theorem}

In \cite{anodai}, 
we characterize an asid bimodule $C$ of $\alpha_{r} =0, \alpha_{\ell} = 0$ 
as a cotilting bimodule over $\Lambda$ in the sense of Miyachi \cite{Miyachi}.

\begin{proposition}[{\cite{anodai}}]\label{cotilting proposition} 
Let $\Lambda$ and $C$ as in Theorem \ref{asid theorem 1}. 
A $\Lambda$-$\Lambda$-bimodule $C$ is 
an asid bimodule of $\alpha_{r} = 0, \alpha_{\ell} = 0$ 
if and only if it is  a $\Lambda$-$\Lambda$-cotilting bimodule. 
\end{proposition}

In \cite{anodai},   
we introduce a new class of finitely graded IG-algebra called \textit{homologically well-graded (hwg) IG-algebra}, 
and show that it posses nice characterizations from several view points. 
Among other things, we show that a trivial extension algebra $A= \Lambda \oplus C$ is hwg IG 
if and only if $C$ is an asid bimodule with $\alpha_{r} = 0, \alpha_{\ell} = 0$.

In \cite{higehaji}, we investigate a condition that $A= \Lambda \oplus C$ is IG from a categorical view point  
in the case where $\Lambda$ is IG. 
We close this section by showing that in this case the above condition become much simpler.

\begin{proposition}\label{reduction of asid conditions}
Let $\Lambda$  and $C$ be as in Theorem \ref{asid theorem 1}.
Assume moreover that $\Lambda$ is IG 
and $C$ satisfies the conditions right and left \asid \ 1.   
Then, $C$ also satisfies the conditions right and left \asid \ 2. 
\end{proposition}

We use the following fundamental observation due to  Iwanaga. 

\begin{proposition}[Iwanaga \cite{Iwanaga}]\label{Iwanaga}
Let $\Lambda$ be an   IG-algebra and $M$ a finitely generated  $\Lambda$-module. Then, 
\[
\pd_{\Lambda} M < \infty \Longleftrightarrow  \injdim_{\Lambda} M < \infty.
\]
\end{proposition}

\begin{proof}[Proof of Proposition \ref{reduction of asid conditions}]
We only prove that  $C$ satisfies the right \asid \ condition 2. 
The  left \asid \ condition 1 together with the left version of Proposition \ref{Iwanaga} 
implies that $\pd C < \infty$. 
Hence by Theorem \ref{Avramov-Foxby:inequality},   
the functor  $\RHom_{\Lambda}(C, -)$ preserves finiteness of injective dimension.   
Therefore $\RHom_{\Lambda}(C^{a}, \Lambda), \ \RHom_{\Lambda} (C^{a}, C)$ 
have finite injective dimension. 
Now the assertion follows from  Lemma \ref{homological dimension lemma}. 
\end{proof}

\section{Upper triangular matrix algebras}\label{Upper triangular matrix algebras} 
 
An important example of  a trivial extension algebra is an upper triangular matrix algebra. 
As applications of Theorem \ref{projective dimension formula} and Theorem \ref{injective dimension formula}, 
we obtain the projective dimension formula and the injective dimension formula for a module over 
an upper triangular matrix algebra. 

Let $A = \begin{pmatrix} \Lambda_{0} & C \\ 0 &\Lambda_{1} \end{pmatrix}$ 
be an upper triangular matrix algebra 
where $\Lambda_{0}, \Lambda_{1}$ are algebras 
and $C$ is a $\Lambda_{0}$-$\Lambda_{1}$-bimodule. 
It is well-known that an $A$-module $M$ can be identified with 
a triple $(M_{0}, M_{1}, \xi_{M})$ consisting of a $\Lambda_{0}$-module $M_{0}$, 
a $\Lambda_{1}$-module $M_{1}$ 
and a $\Lambda_{1}$-module homomorphism $\xi_{M}: M_{0} \otimes_{\Lambda_{0}} C \to M_{1}$. 
We note that $\xi_{M}$ induces a $\Lambda_{0}$-module homomorphism 
$\theta_{M}: M_{0} \to \Hom_{\Lambda_{1}}(C, M_{1})$ via $\otimes$-$\Hom$ adjunction. 
We set 
\[\begin{split}
&\Xi_{M}: M_{0} \lotimes_{\Lambda_{0}}C \xrightarrow{\mathsf{can}} M_{0} \otimes_{\Lambda_{0}} C \xrightarrow{ \ \xi_{M} \ } M_{1}, \\
&\Theta_{M}: M_{0} \xrightarrow{ \ \theta_{M} \ } \Hom_{\Lambda_{1}} (C, M_{1}) 
\xrightarrow{ \ \mathsf{can} \ } \RHom_{\Lambda_{1}}(C, M_{1}) 
\end{split}\]
where $\mathsf{can}$ in both rows denote appropriate canonical morphisms.

\begin{proposition}\label{homological dimension formulas for an upper triangular matrix algebra}
The following equations hold.
\[
\begin{split}
&\pd_{A} M = \sup\{ \  \pd_{\Lambda_{0}} M_{0} , \ 
\pd_{\Lambda_{1}}\cone\Xi_{M} \ \}\\
&\injdim_{A} M = 
\sup\{ \  \injdim_{\Lambda_{1}} M_{1}, \ 
\injdim_{\Lambda_{0}} \cone \Theta_{M}  +1 \  \} 
\end{split}
\]
\end{proposition}

We note that a related  result was obtained by Asadollahi-Salarian \cite[Theorem 3.1, 3.2]{Asadollahi-Salarian}.

\begin{proof}
We set $\Lambda = \Lambda_{0} \times \Lambda_{1}$ 
and $e_{0}:= (1_{\Lambda_{0}}, 0), e_{1} := ( 0, 1_{\Lambda_{1}} )$.
We equip $C$ with  a $\Lambda $-$\Lambda$-bimodule structure 
in the following way
\[
(a_{0},a_{1}) c (b_{0}, b_{1} ) := a_{0} cb_{1}. 
\]
Then the algebra $A$ is the trivial extension algebra $\Lambda\oplus C$. 

Let $M$ be an $A$-module. Then  $M_{0}$ and $M_{1}$ above are obtained as $M_{0} := Me_{0}, M_{1} := Me_{1}$ 
and $M= M_{0} \oplus M_{1}$ can be regarded as a graded $A$-module 
whose degree $0$-part is $M_{0}$ and degree $1$-part is $M_{1}$.  
Now the action map $\tilde{\xi}_{M,0}$ coincides with the above $\xi_{M}$ 
and the coaction map $\tilde{\theta}_{M. 0}$  coincides with the  above $\theta_{M}$. 

Since $C\lotimes_{\Lambda} C= 0, \ M_{1} \lotimes_{\Lambda} C = 0, \ \RHom_{\Lambda}(C, M_{0}) = 0$, 
we deduce the desired formula from Theorem \ref{projective dimension formula} and Theorem \ref{injective dimension formula}.
\end{proof}

Observe that  an $A$-module $M$ fits into a canonical  exact sequence below 
\[
0\to M_{1} \to M \to M_{0} \to 0. 
\]
Hence, $\pd_{A} M \leq \sup\{ \pd_{A} M_{0}, \pd_{A} M_{1} \}$ and 
$\injdim_{A} M \leq \sup\{ \injdim_{A} M_{0}, \injdim_{A} M_{1} \}$. 
Therefore as a corollary, we obtain an answer to Chase's problem 
of determining the global dimension of $A$.  

\begin{corollary}
The following equations hold.
\[
\begin{split}
\gldim A & = 
\sup\{ 
\  \pd_{\Lambda_{0}} M_{0}, \ 
\pd_{\Lambda_{1}} M_{0} \lotimes_{\Lambda_{0}} C + 1, \ 
\pd_{\Lambda_{1}}M_{1} 
 \mid M_{0} \in \Mod \Lambda_{0}, M_{1} \in \Mod \Lambda_{1} \ \}\\
&= 
\sup\{ \  \injdim_{\Lambda_{1}} M_{1}, \ 
\injdim_{\Lambda_{0}}\RHom_{\Lambda_{1}}(C, M_{1}) + 1,  \ 
\injdim_{\Lambda_{0}} M_{0} \ 
 \mid M_{0} \in \Mod \Lambda_{0}, M_{1} \in \Mod \Lambda_{1} \ \}
\end{split}
\]
\end{corollary}

 We  deduce the following well-known result. 
\begin{corollary}
An upper triangular  matrix algebra $A = \begin{pmatrix} \Lambda_{0} & C \\ 0 & \Lambda_{1} \end{pmatrix}$ 
is of finite global dimension 
if and only if so are $\Lambda_{0}, \Lambda_{1}$. 
\end{corollary}

The next result was essentially proved by Enochs-Cortes-Izurdiaga-Torrecillas \cite[Theorem 3.1]{ECIT}.

\begin{corollary}
\begin{enumerate}[(1)]
\item 
Assume that $C_{\Lambda_{1}}$ has finite projective dimension. 
Then an $A$-module $M$ has finite projective dimension 
if and only if 
so do the $\Lambda_{0}$-module $M_{0}$ and the $\Lambda_{1}$-module $M_{1}$. 

\item 
Assume that ${}_{\Lambda_{0}}C$ has finite flat dimension. 
Then an $A$-module $M$ has finite injective dimension 
if and only if 
so do the $\Lambda_{0}$-module $M_{0}$ and the $\Lambda_{1}$-module $M_{1}$. 
\end{enumerate}
\end{corollary}

We can also deduce the following result 
which was obtained by X-W. Chen \cite{Chen scs}. 
We remark that the self-injective dimension of an upper triangular algebra 
was studied by \cite{Sakano}.

\begin{corollary}
Let $\Lambda_{0}, \Lambda_{1}$ be IG-algebras 
and $C$ a $\Lambda_{0}$-$\Lambda_{1}$-bimodule 
which is finitely generated as both left $\Lambda_{0}$-module and right $\Lambda_{1}$-module. 
Then, 
the upper triangular  matrix algebra $A = \begin{pmatrix} \Lambda_{0} & C \\ 0 & \Lambda_{1} \end{pmatrix}$ 
is IG if and only if $\pd {}_{\Lambda_{0}} C < \infty$ and $\pd C_{\Lambda_{1}} < \infty$.  
\end{corollary}

Using Proposition \ref{homological dimension formulas for an upper triangular matrix algebra} repeatedly, 
we obtain the following result which will be used  in the subsequent paper \cite{higehaji} and \cite{anodai}.

\begin{corollary}
Let $A = \bigoplus_{i = 0}^{\ell} A_{i}$ be a finitely graded Noetherian algebra.  
Assume  that the degree $0$-part $A_{0}$ is IG 
and that the $A_{0}$-$A_{0}$-bimodule  $A_{i}$ has finite projective dimension on both sides for $ i = 1, \cdots , \ell$. 
Then the Beilinson algebra $\nabla A$ is IG  and the $\nabla A$-$\nabla A$-bimodule $\Delta A$ has finite projective dimension on both sides.
\end{corollary}

{H.M. Department of Mathematics and Information Sciences,
Faculty of Science / Graduate School of Science,
Osaka Prefecture University}

{minamoto@mi.s.osakafu-u.ac.jp} 

$ $

{K.Y. Graduate Faculty of Interdisciplinary Research, Faculty of Engineering
University of Yamanashi}

{kyamaura@yamanashi.ac.jp}


\begin{thebibliography}{ABCD}
\bibitem{ART}
 Amiot, Claire; Reiten, Idun; Todorov, Gordana, 
The ubiquity of generalized cluster categories. Adv. Math.  226  (2011),  no. 4, 3813-3849

\bibitem{Asadollahi-Salarian}
Asadollahi, Javad;   
Salarian, Shokrollah,
On the vanishing of Ext over formal triangular matrix rings. 
Forum Math.  18  (2006),  no. 6, 951-966. 

\bibitem{AR}
Auslander, Maurice; Reiten, Idun, 
Cohen-Macaulay and Gorenstein Artin algebras.  Representation theory of finite groups and finite-dimensional algebras (Bielefeld, 1991),  221-245. 
Progr. Math., 95, Birkhauser, Basel, 1991. 

\bibitem{ARS} 
 Auslander, Maurice; Reiten, Idun; Smal\o, Sverre O., 
Representation theory of Artin algebras. Corrected reprint of the 1995 original. 
Cambridge Studies in Advanced Mathematics, 36. Cambridge University Press, Cambridge, 1997. xiv+425 

\bibitem{Avramov-Foxby}
Avramov, Luchezar L.; Foxby, Hans-Bj\o rn, Homological dimensions of unbounded complexes. J. Pure Appl. Algebra  71  (1991),  no. 2-3, 129-155.

\bibitem{BIRS}
 Buan, Aslak. B.; Iyama, Osamu; Reiten, Idun; Scott, Jeanne, 
Cluster structures for 2-Calabi-Yau categories and unipotent groups. Compos. Math.  145  (2009),  no. 4, 1035-1079


\bibitem{Buchweitz} 
Buchweitz, Ragnar-Olaf,  
Maximal Cohen-Macaulay Modules and Tate-Cohomology Over Gorenstein Rings, 
unpublished manuscript available at https://tspace.library.utoronto.ca/handle/1807/16682



\bibitem{Chase}
 Chase, Stephen U. 
A generalization of the ring of triangular matrices. Nagoya Math. J.  18  1961 13-25.

\bibitem{Chen scs}
Chen, Xiao-Wu, 
Singularity categories, Schur functors and triangular matrix rings. (English summary) 
Algebr. Represent. Theory  12  (2009),  no. 2-5, 181-191. 

\bibitem{Chen trivial}
Chen, Xiao-Wu,  Graded self-injective algebras ``are" trivial
extensions, {\it J. Algebra} \textbf{322} (2009), 2601-2606



\bibitem{ECIT}
 Enochs, Edgar E.; 
Cortes-Izurdiaga, Manuel; Torrecillas, Blas,  
Gorenstein conditions over triangular matrix rings. J. Pure Appl. Algebra  218  (2014),  no. 8, 1544-1554.


\bibitem{FGR} 
 Fossum, Robert M.; Griffith, Phillip A.; Reiten, Idun, 
Trivial extensions of abelian categories. 
 algebra of trivial extensions of abelian categories with applications to ring theory. 
Lecture Notes in Mathematics, Vol. 456. Springer-Verlag, Berlin-New York, 1975. xi+122 pp.

\bibitem{GLS}
Geiss, Christof
; Leclerc, Bernard
; Schroer, Jan
Quivers with relations for symmetrizable Cartan matrices I: foundations. (English summary) 
Invent. Math.  209  (2017),  no. 1, 61-158. 

\bibitem{GM}
Gelfand, Sergei I.; Manin, Yuri I. Methods of homological algebra. 
Second edition. 
Springer Monographs in Mathematics. 
Springer-Verlag, Berlin, 2003. 
xx+372 pp. ISBN: 3-540-43583-2


\bibitem{Iwanaga}
Iwanaga, Yasuo, 
On rings with finite self-injective dimension. II. 
Tsukuba J. Math.  4  (1980), no. 1, 107-113. 



\bibitem{Happel book} 
Happel, Dieter, 
Triangulated Categories in the Representation Theory of
Finite-Dimensional Algebras. 
London Mathematical  Society  Lecture Notes Series 119.  
Cambridge: 
Cambrdge University Press, 1988.


\bibitem{Happel}
Happel, Dieter, 
On Gorenstein algebras.  Representation theory of finite groups and finite-dimensional algebras (Bielefeld, 1991),  389-404, 
Progr. Math., 95, Birkhauser, Basel, 1991. 





\bibitem{Keller:ddc}
 Keller, Bernhard, Deriving DG categories. Ann. Sci. Ecole Norm. Sup. (4)  27  (1994),  no. 1, 63-102.

\bibitem{Keller:ICM}
Keller, Bernhard, On differential graded categories.  International Congress of Mathematicians. Vol. II,  151-190, Eur. Math. Soc., Zurich, 2006. 

\bibitem{Keller-Reiten}
Keller, Bernhard; Reiten, Idun,
Cluster-tilted algebras are Gorenstein and stably Calabi-Yau,
Adv. Math. 211 (2007), no. 1, 123-151. 




\bibitem{Kimura 1}
Kimura, Yuta, 
Tilting theory of preprojective algebras and $c$-sortable elements, preprint, 
arXiv:1405.4087.

\bibitem{Kimura 2}
Kimura, Yuta, 
Tilting and cluster tilting for preprojective algebras and Coxeter groups, preprint,  
 arXiv:1607.00637. 



\bibitem{Lofwall} 
Lofwall, Clas, 
The global homological dimensions of trivial extensions of rings. 
J. Algebra  39  (1976), no. 1, 287-307. 

\bibitem{MM}
Minamoto, Hiroyuki; 
Mori, Izuru, 
The structure of AS-Gorenstein algebras. Adv. Math. 226 (2011), no. 5, 4061-4095.


\bibitem{higehaji}
 Minamoto, Hiroyuki; Yamaura, Kota,  
On finitely graded Iwanaga-Gorenstein algebras and the stable category their CM-modules, 
in preparation. 


\bibitem{anodai}
Minamoto, Hiroyuki; Yamaura, Kota, 
Homologically well-graded Iwanaga-Gorenstein algebras. in preparation.  



\bibitem{Miyachi}
Miyachi, Jun-ichi, 
Duality for derived categories and cotilting bimodules. 
J. Algebra  185  (1996),  no. 2, 583-603. 



\bibitem{Mori B-construction}
Mori, Izuru, 
B-construction and C-construction. 
Comm. Algebra  41  (2013),  no. 6, 2071-2091. 



\bibitem{NV:Graded and Filtered}
N\u{a}st\u{a}sescu, C.; van Oystaeyen, F. 
Graded and filtered rings and modules. Lecture Notes in Mathematics, 758. Springer, Berlin, 1979



\bibitem{Orlov}
Orlov, Dmitri, 
Derived categories of coherent sheaves and triangulated categories of singularities.  
Algebra, arithmetic, and geometry: in honor of Yu. I. Manin. Vol. II,  503-531, 
Progr. Math., 270, Birkhauser Boston, Inc., Boston, MA, 2009. 


\bibitem{Palmer-Roos}
 Palmer, Ingegerd; Roos, Jan-Erik, 
Explicit formulae for the global homological dimensions of trivial extensions of rings. J. Algebra  27  (1973), 380-413.

\bibitem{Positselski}
Positselski, Leonid,  
Two kinds of derived categories, Koszul duality, and comodule-contramodule correspondence, 
Mem. Amer. Math. Soc. 212 (2011), no. 996, vi+133 pp.

\bibitem{Reiten:Thesis}
Reiten, Idun, 
Trivial extensions and Gorenstein rings, 
Thesis, University of Illinois, Urbana. 




\bibitem{Sakano}
Sakano, Kazunori, 
Injective dimension of generalized triangular matrix rings. 
Tsukuba J. Math.  4  (1980), no. 2, 281-290. 






\bibitem{Yekutieli:Dualizing complexes}
Yekutieli, Amnon, 
Dualizing complexes, Morita equivalence and the derived Picard group of a ring. 
J. London Math. Soc. (2)  60  (1999),  no. 3, 723-746. 


\bibitem{Yekutieli vdB theorem}
Yekutieli, Amnon, 
Another Proof of a Theorem of Van den Bergh about Graded-Injective Modules. 
arXiv:1407.5916



\end{thebibliography}
\end{document}